\documentclass[a4paper,reqno]{amsart}

\usepackage{amsmath,amsfonts,amsthm,amssymb,hyperref,latexsym,mathrsfs,enumerate,pgfplots,tikz-cd}

\DeclareMathOperator{\diag}{\mathrm{diag}}

\DeclareMathOperator{\tr}{\mathrm{tr}}

\DeclareMathOperator{\id}{\mathrm{id}}
\DeclareMathOperator{\Hom}{\mathrm{Hom}}
\DeclareMathOperator{\supp}{\mathrm{supp}}

\DeclareMathOperator{\cu}{\mathrm{Cu}}

\DeclareMathOperator{\cs}{\mathrm{C}^*}
\DeclareMathOperator{\js}{\mathcal{Z}}
\DeclareMathOperator{\aff}{\mathrm{Aff}}

\DeclareMathOperator{\rank}{\mathrm{rank}~}
\DeclareMathOperator{\lip}{\mathrm{Lip}^1}

\newcommand{\nn}{\mathbb{N}}
\newcommand{\zz}{\mathbb{Z}}

\newcommand{\rr}{\mathbb{R}}
\newcommand{\cc}{\mathbb{C}}

\newenvironment{items}{\begin{list} {$\cdot$} {\setlength{\leftmargin}{0.5cm}}}{\end{list}}

\newtheoremstyle{smallcaps}
    {3pt}                    % Space above
    {3pt}                    % Space below
    {\itshape}                   % Body font
    {}                           % Indent amount
    {\sc}                   % Theorem head font
    {.}                          % Punctuation after theorem head
    {.5em}                       % Space after theorem head
    {}  % Theorem head spec (can be left empty, meaning ÔnormalÕ)
    
\newtheoremstyle{smallcapsdef}
    {3pt}                    % Space above
    {3pt}                    % Space below
    {}                   % Body font
    {}                           % Indent amount
    {\sc}                   % Theorem head font
    {.}                          % Punctuation after theorem head
    {.5em}                       % Space after theorem head
    {}  % Theorem head spec (can be left empty, meaning ÔnormalÕ)

\newtheorem {theorem}{Theorem}[section]

\newtheorem {proposition}[theorem]{Proposition}
\newtheorem {corollary}[theorem]{Corollary}
\newtheorem {thm}{Theorem}

\theoremstyle {definition}

\newtheorem {definition}[theorem]{Definition}
\newtheorem {remark}[theorem]{Remark}

\numberwithin{equation}{section}

\title{Metrics on trace spaces}

\author[B.~Jacelon]{Bhishan Jacelon}
\address[B.~Jacelon]{
Institute of Mathematics of the Czech Academy of Sciences\\ \v{Z}itn\'{a} 25\\115 67 Prague 1\\Czech Republic}
\email{bjacelon@gmail.com}


\begin{document}

\maketitle

\begin{abstract}
This article continues the investigation of the tracial geometry of classifiable $\cs$-algebras that have real rank zero and stable rank one. Using the language of optimal transport, we describe several situations in which the distance between unitary orbits of $^*$-homomorphisms into such algebras can be computed in terms of tracial data. The domains we consider are certain (noncommutative) CW complexes, and the measurement is relative to a family of self-adjoint elements that are in a suitable sense tracially Lipschitz. As another application of the utility of this Lipschitz structure, we show how such elements can be repurposed to witness statistical features of endomorphisms in the classifiable category, in particular the tracial version of the (almost-sure) central limit theorem. 
\end{abstract}

\section{Introduction}

This article is a continuation of the work carried out in \cite{Jacelon:2014aa,Jacelon:2021wa}, which address the \emph{Weyl problem for $\cs$-algebras}. In brief, the problem is to identify classes of normal elements of appropriately regular $\cs$-algebras for which the distance between unitary orbits can be computed as the distance between measures on spectra. Typically, `regular' means at least simple, separable, unital, nuclear, $\js$-stable (where $\js$ is the Jiang--Su algebra: the unique classifiable $\cs$-algebra with the same invariant as $\cc$) and sometimes also real rank zero. Often, nuclearity can be relaxed to exactness, and $\js$-stability can be relaxed to pureness (a strictly weaker regularity property of the Cuntz semigroup, introduced in \cite{Winter:2012pi}).

The contribution of \cite{Jacelon:2021wa} was to move from the interval (that is, self-adjoint operators, which were the focus of \cite{Jacelon:2014aa}) to more general spectra (in particular, unitaries), and to an even broader range of (not necessarily planar) commutative domains. The major advancement of the present article is a foray into noncommutativity: we compute the distance between unitary orbits of $^*$-monomorphisms defined on certain one-dimensional NCCW (that is, noncommutative CW) complexes. To avoid obstructions associated with projections, we focus on domains with trivial $K$-theory, namely, prime dimension drop algebras $Z_{p,q}$ and Razak blocks $A_{n,k}$. These $\cs$-algebras are the building blocks of the algebras $\js$ \cite{Jiang:1999hb} and $\mathcal{W}$ \cite{Jacelon:2010fj}, which play central roles in the Elliott classification programme. Without the encumbrance of projections, it is reasonable to expect that the unitary distance between $^*$-monomorphisms from these building blocks into classifiable $\cs$-algebras should admit a measure theoretic computation, and this hope is indeed realized. The following is the content of Theorems~\ref{thm:jiangsu} and \ref{thm:razak}.

\begin{thm} \label{thm1}
Let $A$ be an inductive limit of either prime dimension drop algebras or Razak blocks, with nondegenerate $1$-Lipschitz connecting maps. Let $B$ be an algebraically simple, separable, exact, $\js$-stable $\cs$-algebra. Then,
\[
d_\mathcal{U}(\varphi,\psi) = W_\infty(\varphi,\psi)
\]
for every pair of nondegenerate $^*$-monomorphisms $\varphi,\psi\colon A\to B$.
\end{thm}

Here, the distance $d_\mathcal{U}(\varphi,\psi)$ is relative to a fixed family of $1$-Lipschitz elements of $A$ (see Definitions~\ref{definition:nice}~and~\ref{definition:limit}), and $W_\infty$ is the \emph{$\infty$-Wasserstein distance}, sometimes also called the \emph{optimal matching distance} (\ref{eqn:winf}). Theorem~\ref{thm1} is motivated by \cite[Theorem 4.1]{Jacelon:2014aa}, and its proof is an exercise in applied classification. More precisely, we use both the existence and uniqueness statements of \cite[Theorem 1.0.1]{Robert:2010qy} to transfer the problem from maps $A\to B$ to maps between $Z_{p,q}$'s or $A_{n,k}$'s, where the matter is settled by matching eigenvalues (see Propositions~\ref{prop:dimdrop}~and~\ref{prop:razak}).

Theorem~\ref{thm1} of course applies to single blocks, and in \S\ref{section:examples} we also supply some examples of nontrivial limits. Most interesting of these is a classifiable $\cs$-algebra $A_I$ whose space $T(A_I)$ of tracial states is a Bauer simplex with extreme boundary $\partial_e(T(A_I))$  homeomorphic to $I=[0,1]$. What is additionally noteworthy about $A_I$ is that, by construction, there is a natural metric on $\partial_e(T(A_I))=I$ (in fact, the usual Euclidean metric) and a dense family of Lipschitz observables of the trace space. In other words, the set $\{a \mid \hat a|_{\partial_e(T(A_I))} \text{ is Lipschitz}\}$ of \emph{tracially Lipschitz} elements is dense in the set of self-adjoint elements of $A_I$.

In the setting of topological dynamical systems, Lipschitz observables are useful for witnessing statistical features of chaos, in particular the (almost-sure) central limit theorem (CLT). Foundational examples of chaotic systems are mixing Anosov diffeomorphisms like Arnold's cat map. Many more are described in \cite[\S4]{Jacelon:2022wr}. To say that such a system $(X,\mu,h)$ (where $\mu$ is a Borel probability measure on a metric space $X$, and $h\colon X\to X$ is a $\mu$-preserving measurable map) satisfies the almost-sure CLT for Lipschitz observables is to mean that the following holds for every Lipschitz map $f\colon X\to\rr$. Given such an $f$, let  $\mu_f$ denote its \emph{spatial mean}
\begin{equation} \label{eqn:mean}
\mu_f = \int_X fd\mu
\end{equation}
and $\sigma_f^2$ its \emph{variance}
\begin{equation} \label{eqn:variance}
\sigma_f^2 = \lim_{n\to\infty}\frac{1}{n}\int_X\left(S_nf-n\mu_f\right)^2d\mu,
\end{equation}
where $S_kf$ is the ergodic sum
\[
S_kf=\sum_{i=0}^{k-1}f\circ h^i.
\]
For every $n\in\nn$, let $D_n$ be the normalizing constant $D_n=\sum_{k=1}^n\frac{1}{k}$. Then, the almost-sure CLT holds if, whenever $f\colon X\to\rr$ is Lipschitz with $\mu_f=0$ and  $\sigma_f^2\ne0$ (the former condition arranged by translation and the latter usually being the case), the sequence of weighted averages
\begin{equation} \label{eqn:weight}
T_n(x)=\frac{1}{D_n}\sum_{k=1}^n\frac{1}{k}\delta_{S_kf(x)/\sqrt{k}}
\end{equation}
is $w^*$-convergent to the normal distribution $\mathcal{N}_{0,\sigma_f^2}$, for $\mu$-almost-every $x\in X$.

In the context of $\cs$-dynamics, if such an $X$ is the extreme boundary of the trace space of a suitable classifiable $\cs$-algebra $A_X$, then the preserved measure $\mu$ corresponds to a trace $\tau$, and we can use classification to lift $h$ to a $\tau$-preserving endomorphism $\theta_h\colon A_X\to A_X$. The almost-sure CLT can then be interpreted tracially. In \S\ref{section:dynamics}, we apply this to one of the simplest examples of a strongly chaotic system, an expanding circle map $h\colon S^1\to S^1$ (specifically, a Pomeau--Manneville type system, used to model intermittent turbulence). The classifiable $\cs$-algebra $A_{S^1}$ of Theorem~\ref{thm2} is constructed in exactly the same way as $A_I$, but with dimension drop algebras over the interval replaced by generalized dimension drop algebras over the circle. (For a generalization of this construction, and a fuller treatment of tracially chaotic endomorphisms of classifiable $\cs$-algebras, see \cite{Jacelon:2022wr}.)

\begin{thm} \label{thm2}
There exists a simple, separable, unital, nuclear, $\js$-stable, projectionless $\cs$-algebra $A=A_{S^1}$ that has trivial tracial pairing and satisfies the UCT, such that $\partial_e(T(A))\cong S^1$ and the set $\{a \mid \hat a|_{S^1} \text{ is Lipschitz}\}$ (with respect to the geodesic metric on $S^1$) of tracially Lipschitz elements is dense in $A_{sa}$. Moreover, there is a trace $\tau_0\in T(A)$ and a $\tau_0$-preserving endomorphism $\theta$ of $A$ such that, for every tracially Lipschitz $a$ with $\tau_0(a)=0$ and $\sigma_{\hat{a}}^2>0$ on $\partial_e(T(A))$, and almost every $\tau\in\partial_e(T(A))$, the sequence of weighted averages (in the sense of (\ref{eqn:weight})) of the point masses
\[
\left\{\delta_{\frac{1}{\sqrt{k}}\tau(a+\theta(a)+\dots+\theta^{k-1}(a))}\right\}_{k=1}^n
\]
is $w^*$-convergent to $\mathcal{N}_{0,\sigma_{\hat{a}}^2}$.
\end{thm}

In fact, the usual CLT (involving convergence in distribution rather than almost-sure convergence, and in fact valid for observables that are merely H\"{o}lder continuous) also holds and can be interpreted at the level of the $\cs$-algebra. We are here presenting the almost-sure version to emphasize a result requiring observables to be suitably Lipschitz (see for example \cite[Theorem 2.19]{Chazottes:2007wg}).

The second task of this article is to answer some natural questions arising from \cite{Jacelon:2021wa}. As alluded to in \cite[Remark 4.13]{Jacelon:2021wa}, one begins to develop the feeling that optimal unitary conjugation between spectrum-sharing normal elements of classifiable $\cs$-algebras might depend on that spectrum's geometry. The idea of \emph{continuous transport}, a property satisfied by a compact path-connected metric space $(X,d)$ provided that it is in a suitable sense sufficiently uniform, is really an attempt to hone in on this intuition. Precisely, it means that any two faithful and diffuse measures can be mapped one onto the other by a homeomorphism $h$ (called a \emph{transport map}) whose distance from the identity is at most the $W_\infty$ distance (\ref{eqn:winf}) between the measures. Its utility is realized via classification, which allows us to translate continuous transport of measures into optimal unitary conjugation, yielding a version of Theorem~\ref{thm1} for $^*$-monomorphisms from $C(X)$ into tracial classifiable $\cs$-algebras of real rank zero (with some restrictions on $K$-theory and traces). The questions are:
\begin{enumerate}[(I)]
\item \label{q1} To what extent can the assumptions on $K$-theory and traces be relaxed?
\item \label{q2} What are examples of spaces $X$ with this property?
\item \label{q3} Are there $W_p$-versions of the theorem for $1\le p<\infty$?
\end{enumerate}

Here, $\{W_p\}_{p\in[1,\infty]}$ are the \emph{$p$-Wasserstein distances} (\ref{eqn:wass}); they are the titular metrics of this article. If `metric' is understood to mean one that induces the $w^*$-topology on the space $\mathcal{M}_f(X)$ of faithful Borel probability measures on $X$, then the inclusion of $W_\infty$ is justified by Proposition~\ref{prop:bottleneck}.

To answer Question~\ref{q1}, we first quantize the notion of continuous transport in the form of the transport constant $k_X$ (\ref{eqn:consant2}). While $k_X=1$ for every example $X$ shown in \cite{Jacelon:2021wa} to admit (approximate) continuous transport (such as the circle and compact convex subsets of Euclidean space), if $X$ is for example a noncircular ellipse, then $k_X>1$. As for Question~\ref{q2}, we observe in Theorem~\ref{thm:manifolds} that the arguments of \cite{Jacelon:2021wa} extend naturally to the setting of higher dimensional compact, connected Riemannian manifolds. Lie groups are particularly attractive targets, especially those like $\mathrm{U}(n)$, $\mathrm{SU}(n)$ and $\mathrm{Sp}(n)$ that have (finitely generated and) torsion-free $K$-theory, which is now our only assumption on $K_*(C(X))$ (replacing the `$K$-planarity' assumption of \cite[Theorem 4.11]{Jacelon:2021wa}). As for tracial assumptions, we take advantage of recent classification \cite{Lin:2014aa,Elliott:2016ab,Gong:2020uf} (or \cite{Carrion:wz}) to weaken the requirement that $\partial_e(T(A))$ be finite, more generally allowing that it be compact and of finite Lebesgue covering dimension. Finally, under these assumptions we answer Question~\ref{q3} when $X=[0,1]$. In this case, $W_p$ can replace $W_\infty$ as long as the unitary distance is computed with respect to the appropriate tracial Schatten $p$-norm (\ref{eqn:unitary}).

The following is the combination of Theorems~\ref{thm:bauer}~and~\ref{thm:interval}.

\begin{thm} \label{thm3}
Let $X$ be a compact, path-connected metric space such that $k_X<\infty$ and $K^*(X)$ is finitely generated and torsion free. Let $A$ be a simple, separable, unital, nuclear, $\js$-stable $\cs$-algebra of real rank zero, such that the extreme boundary $\partial_e(T(A))$ of its tracial state space is nonempty, compact and of finite Lebesgue covering dimension. Then,
\begin{equation} \label{eqn:live}
W_\infty(\varphi,\psi) \le d_\mathcal{U}(\varphi,\psi) \le k_X \cdot W_\infty(\varphi,\psi)
\end{equation}
for every pair of unital $^*$-monomorphisms $\varphi,\psi\colon C(X)\to A$ with $K_*(\varphi)=K_*(\psi)$. In the special case $X=[0,1]$, for every such $\varphi,\psi\colon C([0,1])\to A$ and for any $p\in[1,\infty]$,
\begin{equation} \label{eqn:die}
d_{\mathcal{U},p}(\varphi,\psi) = W_p(\varphi,\psi),
\end{equation}
with no assumptions about the real rank or tracial structure of $A$ needed if $p=\infty$.
\end{thm}

It should be noted that the proofs of the two parts of Theorem~\ref{thm3} are rather different. For (\ref{eqn:live}), we solve a transport problem on the space $X$ to the extent allowed by its geometry, then use classification to interpret the transport map as a conjugating unitary in the codomain algebra $A$. For the $p=\infty$ case of (\ref{eqn:die}), specialization to $X=[0,1]$ affords us access to powerful Cuntz semigroup classification that we use to solve the transport problem within $A$ directly. 

While we have spotlighted Theorem~\ref{thm1} as the principal novelty of this article, a precise treatment of it (in particular, the definition of the distances $d_\mathcal{U}$ and $W_\infty$) is somewhat technical. Since the motivation comes from the commutative setting, that is where we begin the story and also how we shape our narrative: We start with the domain $C(X)$, then collapse in spatial dimension to $X=[0,1]$, and finally expand in fibre dimension to arrive at one-dimensional NCCW complexes. These noncommutative domains are covered by the classifying Cuntz functor in the same way as is $C([0,1])$, and so the same techniques ultimately get us to both (\ref{eqn:die}) and Theorem~\ref{thm1}.

This article is therefore organized as follows. First, we give in \S\ref{section:wasserstein} the definitions and basic properties of the Wasserstein metrics $\{W_p\}_{p\in[1,\infty]}$ and the transport constant $k_X$, and we show that $k_X=1$ if $X$ is a compact, connected Riemannian manifold of dimension at least three. \S~\ref{section:bauer} contains the first part (\ref{eqn:live}) of Theorem~\ref{thm3}. The second part (\ref{eqn:die}), as well as Theorem~\ref{thm1}, is the focus of \S~\ref{section:onedim}. Finally, \S~\ref{section:dynamics} is devoted to Theorem~\ref{thm2}.

\subsection*{Acknowledgements} This research was supported by the GA\v{C}R project 20-17488Y and RVO: 67985840. I am grateful to Karen Strung and Alessandro Vignati for conversations about dimension drop algebras held during my visit to the Institut de Math\'ematiques de Jussieu-Paris Rive Gauche in November 2021, to Karmen Grizelj, Andrey Krutov and R\'{e}amonn \'{O} Buachalla for chats about Lie groups, and to the anonymous referees whose suggestions have helped improve the present exposition.

\section{The Wasserstein metrics and continuous transport} \label{section:wasserstein}

Throughout the article, $(X,d)$ is a compact, path-connected metric space. We denote by $\mathcal{M}(X)$ the set of Borel probability measures on $X$, by $\mathcal{M}_f(X)$ those measures that are faithful (that is, fully supported) and by $\mathcal{M}_g(X)$ those that are faithful and also diffuse (that is, atomless). We write
\[
\lip(X) = \{f\colon X\to\mathbb{R} \mid \forall\,x,y\in X\, (|f(x)-f(y)| \le d(x,y))\}.
\]

\subsection{The Wasserstein metrics}

Motivated by the discrete setting, the distance
\begin{equation} \label{eqn:winf}
\delta(\mu,\nu) = \inf\{r>0 \mid\forall\, U\subseteq X\text{ Borel } (\mu(U) \le \nu(U_r))\}
\end{equation}
was referred to in \cite{Jacelon:2021wa} as the \emph{optimal matching distance}. (Here, $U_r$ denotes the $r$-neighbourhood $\{x\in X \mid d(x,U)<r\}$ of $U$.) However, in the world of geometric measure theory it may be more commonly recognized as the \emph{$\infty$-Wasserstein} distance $W_\infty$ (see \cite[Proposition 5]{Givens:1984to}, from which one also obtains the symmetry of the definition). It dominates the \emph{L\'{e}vy--Prokhorov metric}
\begin{equation} \label{eqn:lp}
d_P(\mu,\nu) = \inf\{r>0 \mid\forall\, U\subseteq X\text{ Borel } (\mu(U) \le \nu(U_r) + r)\}
\end{equation}
and is the right distance to use in the study of norm-closed unitary orbits in $\cs$-algebras. Although (unlike $d_P$) $W_\infty$ yields a strictly finer topology than the $w^*$-topology on the full space of measures $\mathcal{M}(X)$, it \emph{does} give the $w^*$-topology on the measures of interest to us here (see Proposition~\ref{prop:bottleneck} below).

By contrast, for $1\le p<\infty$ the \emph{$p$-Wasserstein distance}
\begin{equation} \label{eqn:wass}
W_p(\mu,\nu) = \inf_{\pi\in\Pi(\mu,\nu)} \left(\int_{X\times X} d(x,y)^pd\pi(x,y)\right)^\frac{1}{p}
\end{equation}
(where $\Pi(\mu,\nu)$ denotes the set of Borel probability measures on $X\times X$ with marginals $\mu$ and $\nu$) \emph{does} provide a metrization of the $w^*$-topology on all of $\mathcal{M}(X)$ (see \cite[Proposition 4]{Givens:1984to}). To make use of them, however, we must replace the $\cs$-norm $\|\cdot\|=\|\cdot\|_\infty$ by the \emph{uniform tracial Schatten $p$-norm}
\begin{equation} \label{eqn:schatten}
\|a\|_p = \sup_{\tau\in T(A)} \tau(|a|^p)^{\frac{1}{p}}
\end{equation}
and measure the unitary distance between $^*$-homomorphisms $\varphi,\psi\colon C(X)\to A$ as
\begin{equation} \label{eqn:unitary}
d_{\mathcal{U},p}(\varphi,\psi) = \inf_{u\in \mathcal{U}(\tilde{A})} \sup_{f\in\lip(X)} \|u\varphi(f)u^*-\psi(f)\|_p.
\end{equation}
Correspondingly, we write
\begin{equation} \label{eqn:wp}
W_p(\varphi,\psi) = \sup_{\tau\in T(A)}W_p(\mu_{\varphi^*\tau},\mu_{\psi^*\tau})
\end{equation}
for any $p\in[1,\infty]$. Here, given a $\cs$-algebra $A$, $T(A)$ denotes the space of tracial states on $A$, $\tilde{A}$ denotes the minimal unitization of $A$, and $\mathcal{U}(\tilde{A})$ is its unitary group. Going forward, we will also denote by $\tau\mapsto\mu_\tau$, $\mu\mapsto\tau_\mu$ the natural inverse affine homeomorphisms between $T(C(X))$ and $\mathcal{M}(X)$.

\begin{remark} \label{remark:representations}
In some cases, there are other useful descriptions of $W_p$.
\begin{enumerate}[(i)]
\item \label{rep1} ($X\subseteq\mathbb{R}$, $p\in[1,\infty)$) Let $F(x)=\mu(-\infty,x]$ and $G(x)=\nu(-\infty,x]$ be the cumulative distribution functions of $\mu$ and $\nu$. Their inverses are
\[
F^{-1}(t) = \inf\{x\in X \mid F(x) > t\}, \quad G^{-1}(t) = \inf\{x\in X \mid G(x) > t\}.
\]
Then, as in for example \cite[Theorem 5.1]{Ambrosio:2003vk} (see also the description of `some famous couplings' in Chapter 1 of \cite{Villani:2009aa}) with $\varphi(x)=x^p$ and the substitution $t=F(x)$,
\[
W_p(\mu,\nu) = \left(\int_0^1 |F^{-1}(t)-G^{-1}(t)|^p dt\right)^\frac{1}{p} = \|F^{-1}(t)-G^{-1}(t)\|_p.
\]
Equivalently, in the notation of \cite[\S3]{Hiai:1989aa}, $W_p(\mu,\nu) = \|\lambda(\mu)-\lambda(\nu)\|_p$,
where for $t\in[0,1)$, $\lambda_t(\mu) = \inf\{s\in\mathbb{R} \mid \mu(s,\infty)\le t\}$ and similarly for $\lambda_t(\nu)$.

In particular, if $\mu=\frac{1}{n}\sum_{i=1}^n\delta_{x_i}$ and $\nu=\frac{1}{n}\sum_{i=1}^n\delta_{y_i}$ with $x_1<\dots<x_n$ and $y_1<\dots<y_n$, then
\[
W_p(\mu,\nu) = \left(\frac{1}{n}\sum_{i=1}^n |x_i-y_i|^p\right)^\frac{1}{p}.
\]
\item \label{rep2} ($p=1$) By the Kantorovich--Rubinstein Theorem \cite{Kantorovic:1957we},
\[
W_1(\mu,\nu) = \sup\left\{\left|\int_X fd\mu - \int_X fd\nu\right| \mid f\in\lip(X)\right\}.
\]
For the general duality formula, see \cite[Theorem 3.1]{Ambrosio:2003vk}.
\item \label{rep3} In general (see \cite[Proposition 3]{Givens:1984to}), $W_p(\mu,\nu) \le W_q(\mu,\nu)$ whenever $p \le q$, and \[\lim_{p\to\infty}W_p(\mu,\nu)=W_\infty(\mu,\nu).\]
\end{enumerate}
\end{remark}

\begin{proposition} \label{prop:bottleneck}
For every compact, connected metric space $X$, the $W_\infty$-topology coincides with the $w^*$-topology on $\mathcal{M}_f(X)$.
\end{proposition}

\begin{proof}
We will provide the proof for $X=[0,1]$ (with the Euclidean metric), which is what is being referred to in Remark~\ref{remark:lipschitz}(\ref{lip3}). See \cite{mathoverflow/steve:tm} for a general argument.

Let $\varepsilon\in(0,1)$ and $\mu\in\mathcal{M}_f([0,1])$. We will find $\gamma>0$ such that $W_\infty(\mu,\nu)<\varepsilon$ for every $\nu\in\mathcal{M}_f([0,1])$ with $W_1(\mu,\nu)<\gamma$. Choose a natural number $N>\frac{12}{\varepsilon}$, and define subintervals $(U_i)_{0\le i\le N+1}$ and $(V_i)_{1\le i\le N}$ by $U_0=U_{N+1}=\emptyset$, and
\[
U_i=\left[\frac{i-1}{N},\frac{i}{N}\right] \quad,\quad V_i=U_{i-1}\cup U_i \cup U_{i+1} \quad \text{for } 1\le i\le N.
\]
Set
\[
r=\frac{1}{2}\min\left\{\varepsilon,\min_{1\le 1\le N}\mu(U_i)\right\} \quad \textrm{and} \quad \gamma = \frac{r^2}{N},
\]
and let $\nu\in\mathcal{M}_f([0,1])$ with $W_1(\mu,\nu)<\gamma$. For $1\le i\le N$, let $f_i\in C([0,1])^1_+$ be an $N$-Lipschitz function supported on $V_i$ that is constantly 1 on $U_i$. Then,
\begin{equation} \label{mintrace}
\nu(V_i) \ge \int f_i d\nu > \int f_i d\mu - NW_1(\mu,\nu) > \mu(U_i) - r \ge r.
\end{equation}
Now let $U\subseteq[0,1]$ be an arbitrary open subset. By \cite[Theorem 2]{Gibbs:2002aa},
\begin{equation} \label{gibbs}
d_P(\mu,\nu)^2 \le W_1(\mu,\nu),
\end{equation}
where $d_P$ is the L\'{e}vy--Prokhorov metric, so by definition (\ref{eqn:lp}) of $d_P$,
\begin{equation} \label{prokhorov}
\mu(U) < \nu(U_r) + r \quad \textrm{ and } \quad \nu(U) < \mu(U_r) + r.
\end{equation}
If $U_\frac{\varepsilon}{2}$ intersects every $V_i$, each of which has length $<\frac{\varepsilon}{4}$, then $U_\varepsilon \supseteq (U_\frac{\varepsilon}{2})_\frac{\varepsilon}{2} = [0,1]$, and so we certainly have
\[
\nu(U) \le \mu(U_\varepsilon) \quad \textrm{ and } \quad \mu(U) \le \nu(U_\varepsilon).
\]
Otherwise, there exists $i$ such that $U_\frac{\varepsilon}{2}$ does not intersect $V_i$ but does intersect at least one of its neighbours. By (\ref{mintrace}) and (\ref{prokhorov}), we then have
\[
\nu(U_\varepsilon) \ge \nu(U_\frac{\varepsilon}{2})_\frac{\varepsilon}{2} \ge \nu(U_\frac{\varepsilon}{2}) + \nu(V_i) \ge \nu(U_\frac{\varepsilon}{2}) + r \ge  \nu(U_r) + r > \mu(U),
\]
and similarly the other way round.
\end{proof}

\begin{proposition} \label{prop:lipschitz}
Let $X$ be a compact metric space. Then, for any $1\le p\le \infty$ and $\mu,\nu\in\mathcal{M}(X)$,
\[
\sup_{f\in\lip(X)}W_p(f_*\mu,f_*\nu) \le W_p(\mu,\nu)
\]
(where $f_*\mu$ denotes the pushforward $\mu\circ f^{-1}$). Moreover, equality holds for $X=[0,1]$.
\end{proposition}

\begin{proof}
Fix $\pi\in\Pi(\mu,\nu)$ and $f\in\lip(X)$. Then $(f\times f)_*\pi\in\Pi(f_*\mu,f_*\nu)$, which simply means the following: for any open set $U\subseteq f(X)$,
\begin{align*}
(f\times f)_*\pi(U\times f(X)) &= \pi((f\times f)^{-1}(U\times f(X))) &&(\text{definition of pushforward})\\
&= \pi(f^{-1}(U)\times X)\\
&= \mu(f^{-1}(U)) &&(\text{definition of } \pi\in\Pi(\mu,\nu))\\
&= f_*\mu(U) 
\end{align*}
and similarly $(f\times f)_*\pi(f(X)\times V)=f_*\nu(V)$ for any open $V\subseteq f(X)$. Hence,
\begin{align*}
W_p(f_*\mu,f_*\nu)^p &= \inf_{\gamma\in\Pi(f_*\mu,f_*\nu)} \int_{f(X)\times f(X)} |s-t|^pd\gamma(s,t)\\
&\le \int_{f(X)\times f(X)} |s-t|^pd(f\times f)_*\pi(s,t)\\
&= \int_{X\times X} |f(x)-f(y)|^pd\pi(x,y)\\
&\le \int_{X\times X} d(x,y)^pd\pi(x,y),
\end{align*}
and so $W_p(f_*\mu,f_*\nu)^p \le W_p(\mu,\nu)^p$. The case $p=\infty$ follows upon taking the limit $p\to\infty$ (see Remark~\ref{remark:representations}(\ref{rep3})) but is also proved directly in \cite[Proposition 3.4]{Jacelon:2021wa}. The second statement holds since $\id\in\lip([0,1])$.
\end{proof}

For $X=[0,1]$, the relationship between $W_p$ and $d_{\mathcal{U},p}$ is provided by Proposition~\ref{prop:wasserstein}, which describes the prototypical example of a transport map, and the following corollary of Proposition~\ref{prop:lipschitz}.

\begin{corollary} \label{corollary:interval}
\sloppy
Let $A$ be a simple, tracial, unital $\cs$-algebra, let $\varphi,\psi\colon C([0,1])\to A$ be unital $^*$-monomorphisms, and let $p\in[1,\infty]$. Then,
\[
W_p(\varphi,\psi) \le d_{\mathcal{U},p}(\varphi,\psi).
\]
\fussy
\end{corollary}

\begin{proof}
The proof is the same as that of \cite[Corollary 3.6]{Jacelon:2021wa}, using Proposition~\ref{prop:lipschitz} instead of \cite[Proposition 3.4]{Jacelon:2021wa}, and \cite[Theorem 4.3]{Hiai:1989aa} instead of \cite[Theorem 2.1]{Hiai:1989aa} for the required version of \cite[Lemma 3.3(ii)]{Jacelon:2021wa} (stated for positive rather than commuting normal elements).
\end{proof}

\begin{proposition} \label{prop:wasserstein}
Let $\mu,\nu\in\mathcal{M}_g([0,1])$ with cumulative distribution functions $F,G$, and let $p\in[1,\infty]$. Then, the increasing rearrangement homeomorphism $h=F^{-1}\circ G\colon [0,1]\to[0,1]$ satisfies $h_*(\nu)=\mu$ and $\|h-\id\|_p \le W_p(\mu,\nu)$, where the $p$-norm is taken in $L^p([0,1],\nu)$.
\end{proposition}

\begin{proof}
See \cite[Theorem 5.1]{Ambrosio:2003vk}, and also \cite[Proposition 2.2]{Jacelon:2021wa} (for the $p=\infty$ case). 
\end{proof}

\subsection{Continuous transport}

\begin{definition} \label{def:transport}
Given $\mu,\nu\in\mathcal{M}(X)$ and $\varepsilon>0$, write $\mathcal{H}(\nu,\mu,\varepsilon)$ for the set of homeomorphisms $h\colon X\to X$ with $W_\infty(h_*\nu,\mu)<\varepsilon$ (where $h_*\nu$ is the pushforward measure $\nu\circ h^{-1}$). Say that $X$ \emph{approximately admits continuous transport} if, for every $\mu,\nu\in\mathcal{M}_g(X)$ and $\varepsilon>0$, there exists $h\in\mathcal{H}(\nu,\mu,\varepsilon)$ such that
\[
d(h,\id)=\sup_{x\in X}d(h(x),x)<W_\infty(\mu,\nu)+\varepsilon.
\]
Equivalently, the value of the constant
\begin{equation} \label{eqn:constant1}
c_X = \sup_{\nu\ne\mu\in\mathcal{M}_g(X)}\sup_{\varepsilon>0}\inf_{h\in\mathcal{H}(\nu,\mu,\varepsilon)}\frac{d(h,\id)}{W_\infty(\mu,\nu)}
\end{equation}
is $1$. For applications, we would like the transport homeomorphisms $h$ to be trivial on $K$-theory, so let us write $k_X$ for the potentially larger \emph{transport constant}
\begin{equation} \label{eqn:consant2}
k_X = \sup_{\nu\ne\mu\in\mathcal{M}_g(X)}\sup_{\varepsilon>0}\inf_{h\in\mathcal{H}_0(\nu,\mu,\varepsilon)}\frac{d(h,\id)}{W_\infty(\mu,\nu)},
\end{equation}
where $\mathcal{H}_0(\nu,\mu,\varepsilon)=\{h\in\mathcal{H}(\nu,\mu,\varepsilon) \mid h \,\text{is homotopic to}\,\id\}$.
\end{definition}

The transport constant is most likely to provide meaningful information about $(X,d)$ when the metric agrees with the intrinsic one (given by the infimum of the lengths of paths joining points). This is the case for the Riemannian manifolds considered below, but not necessarily for curves embedded in the plane with the inherited Euclidean metric (which are of particular interest as they represent potential spectra of normal elements of $\cs$-algebras). As an example, a circular arc $X$ has $k_X=1$ if $X$ is smaller than a semicircle, but $k_X$ grows as $X$ approaches the full circle (because the $W_\infty$ distance between measures concentrated at the endpoints of the arc converges to $0$, while the distance between an associated transport map and the identity does not). That said, in the real rank zero setting a much broader range of spectra $X\subseteq\mathbb{C}$ is accessible; see for example \cite{Hu:2015aa} and \cite[Theorem 4.10]{Jacelon:2021wa} (which allows for arbitrary Peano continua, and can be generalized to the class of $\cs$-algebras considered in Theorem~\ref{thm:bauer} below).

\sloppy
\begin{remark}
\begin{enumerate}[(i)]
\item It is immediate from the definition that $c_X\ge 1$ for any $X$, and by the Oxtoby--Ulam Theorem \cite{Oxtoby:1941aa}, $c_X<\infty$ if $X$ is a topological manifold.
\item The same value of $c_X$ or $k_X$ is obtained whether one takes $p=\infty$ or any other $p\in[1,\infty]$ in the definition of $\mathcal{H}(\nu,\mu,\varepsilon)$, because all of the Wasserstein metrics $W_p$ are topologically equivalent on $\mathcal{M}_f(X)$ (see Proposition~\ref{prop:bottleneck}).
\item For the purposes of this article, it is usually enough for the transport map $h$ to be a continuous surjection, or in other words for the induced map $h^*\colon C(X)\to C(X)$ to be injective. This guarantees that $\psi\circ h^*\colon C(X)\to A$ is a $^*$-monomorphism whenever $\psi\colon C(X)\to A$ is. That said, in each example we consider it is no more difficult to in fact obtain a bijection, so we have chosen to incorporate this into the definition.
\item The $W_\infty$-closure of the set of finitely supported measures on $X$ includes all of $\mathcal{M}_f(X)$. (See \cite[Lemma 2.3]{Jacelon:2021wa}, which is stated for $\mathcal{M}_g(X)$ but whose proof does not actually use diffuseness.) In most examples where we are able to compute $c_X$ or $k_X$, we do so by approximating $\mu$ and $\nu$ by $\mu'=\frac{1}{n}\sum_{i=1}^n\delta_{x_i}$ and $\nu'=\frac{1}{n}\sum_{i=1}^n\delta_{y_i}$, exhibiting a homeomorphism $h$ with $h_*\nu'=\mu'$, and observing an \emph{a priori} bound for the Lipschitz constant $L$ of $h$ in terms of $\inf_{\sigma\in S_n}d(x_i,y_{\sigma(i)})=W_\infty(\mu',\nu')\approx W_\infty(\mu,\nu)$. The estimate
\begin{align*}
W_\infty(h_*\nu,\mu) &\le W_\infty(h_*\nu,h_*\nu') + W_\infty(h_*\nu',\mu') + W_\infty(\mu',\mu) \\
&\le L\cdot W_\infty(\nu,\nu') + W_\infty(\mu',\mu)
\end{align*}
allows us to conclude that $h\in\mathcal{H}(\nu,\mu,\varepsilon)$. This is the strategy employed in Proposition 2.5 and Theorem 2.13 of \cite{Jacelon:2021wa}, and implicitly in Theorem~\ref{thm:manifolds} below.
\end{enumerate}
\end{remark}
\fussy

In \cite[Theorem 2.13]{Jacelon:2021wa}, we showed that a compact convex subset $X$ of Euclidean space has transport constant $c_X=k_X=1$. The key property of such a space used in the proof is that points can be joined by paths witnessing the distance between them (that is, straight lines). Riemannian manifolds also enjoy this property.

\begin{theorem} \label{thm:manifolds}
Let $X$ be a compact, connected Riemannian manifold of dimension $\ge 3$ equipped with its intrinsic metric $d$. Then, $k_X=1$, that is, for every $\mu,\nu\in\mathcal{M}_g(X)$ and every $\varepsilon>0$, there exists a homeomorphism $h\colon X\to X$ homotopic to the identity such that $W_\infty(h_*\nu,\mu)<\varepsilon$ and $d(h,\id)<W_\infty(\mu,\nu)+\varepsilon$.
\end{theorem}

\begin{proof}
This follows from \cite[Lemma 2.3]{Jacelon:2021wa} (which shows $W_\infty$-density of finitely supported measures) and the argument of \cite[Propositions 2.8]{Jacelon:2021wa} (which shows how to transport one finitely supported measure onto another). The only difference is that straight lines are replaced by length-minimizing geodesics, which exist by the Hopf--Rinow Theorem (see \cite[\S5.3]{Carmo:1976um}). The constructed homeomorphism $h$ is equal to the identity except within finitely many disjoint tubular neighbourhoods of paths, within which $h$ may be continuously deformed to the identity.
\end{proof}

\begin{remark} \label{remark:twodim}
The case $\dim X=2$ is rather different, because we lack the extra dimension to locally perturb intersecting geodesics into nonintersecting paths. That said, the argument of \cite[Propositions 2.9]{Jacelon:2021wa} does carry over to measures on the sphere $S^2$ (equipped with spherical distance), and also on surfaces of nonzero genus provided that the $W_\infty$ distance is small compared to the surface's systole (that is, the length of the shortest closed homotopically nontrivial geodesic).
\end{remark}

\begin{remark} \label{remark:lie}
Theorem~\ref{thm:manifolds} applies in particular to compact, connected Lie groups $G$ (see for example \cite[Theorem 3.8]{Arvanitoyeorgos:2003vq}; in the semisimple case the metric is provided by the Killing form). Lie groups have many properties that make them especially attractive manifolds in the context of this article. In particular, if $\pi_1(G)$ is torsion free (for example, if $G$ is simply connected), then $K^*(G)$ is torsion free. (If $G$ is also semisimple, $K^*(G)$ is in fact isomorphic as a Hopf algebra over the integers to the exterior algebra generated by the $K^1$-classes of the fundamental representations of $G$; see \cite{Hodgkin:1967kq}). Homogeneous spaces $G/K$ associated to such a group $G$ also have torsion free $K$-groups (see \cite{Minami:1975wx}). This simplifies the analysis of $^*$-homomorphisms from $C(X)$ in the sense that $KL(C(X),A)$ (discussed for example in \cite[\S2]{Matui:2011uq}) becomes simply $\Hom(K_*(C(X)),K_*(A))$ for any $\sigma$-unital $\cs$-algebra $A$.
\end{remark}

\section{Optimal unitary conjugation} \label{section:bauer}

In this section, we relax some of the tracial and $K$-theoretic constraints imposed on the codomains $A$ considered in \cite[Theorem 4.11]{Jacelon:2021wa}, and potentially allow for transport constants $k_X>1$. This would apply for example to ellipses $X\subseteq\mathbb{C}$, although Theorem~\ref{thm:bauer} might not be optimal in this case.

\begin{theorem} \label{thm:bauer}
Let $X$ be a compact, path-connected metric space such that $k_X<\infty$ and $K^*(X)$ is finitely generated and torsion free. Let $A$ be a simple, separable, unital, nuclear $\js$-stable $\cs$-algebra of real rank zero, such that the extreme boundary $\partial_e(T(A))$ of its tracial state space is nonempty, compact and of finite Lebesgue covering dimension. Then,
\[
W_\infty(\varphi,\psi) \le d_\mathcal{U}(\varphi,\psi) \le k_X \cdot W_\infty(\varphi,\psi)
\]
for every pair of unital $^*$-monomorphisms $\varphi,\psi\colon C(X)\to A$ with $K_*(\varphi)=K_*(\psi)$.
\end{theorem}

\begin{proof}
The overall strategy of proof is the same as that of \cite[Theorem 4.11]{Jacelon:2021wa}. As in \cite[Proposition 4.9]{Jacelon:2021wa}, we approximately diagonalize $\psi$:
\begin{equation} \label{eqn:diagonal}
\begin{tikzcd}
C(X) \arrow[rr,dashed,"\psi' "] \arrow[dr,"\bigoplus_{i=1}^m\psi_i"] & & A\\
& \bigoplus_{i=1}^m p_iAp_i \arrow[ur, hook] & 
\end{tikzcd}
\end{equation}
with the projections $p_i$ corresponding to a suitable partition of unity on $\partial_e(T(A))$ and the measures $\mu_{\psi_i^*\tau}$ faithful and diffuse for every $\tau\in T(A)$ (as in \cite[Proposition 4.7]{Jacelon:2021wa}). These measures are then transported to their $\varphi$-counterparts to the extent that the space $X$ allows (represented by the constant $k_X$). Finally, classification delivers the required conjugating unitary.

That said, the present increased level of generality does introduce technical subtleties that must be addressed, especially in the diagonalization step. For completeness, we include full details at least up to that step, then indicate how to conclude the argument from there.

Since $A$ is nuclear (hence exact) and has real rank zero, the state space of the simple ordered (weakly unperforated) abelian group $(K_0(A),K_0(A)_+,[1_A])$ is affinely homeomorphic to $T(A)$, so in particular is a metrizable Choquet simplex. By \cite[Proposition 5.8]{Lin:2011ub}, $(K_0(A),K_0(A)_+,[1_A])$ therefore has the `rationally Riesz' property, and so by the range result \cite[Theorem 6.8]{Lin:2011ub} combined with the classification of unital, simple, separable, nuclear, $\mathcal{Z}$-stable $\cs$-algebras satisfying the UCT (see \cite[Theorem 29.8]{Gong:2020uf} and \cite[Theorem 4.9]{Elliott:2016ab}), $A$ has rational tracial rank at most one. We can then apply (the approximate version of) \cite[Corollary 5.4]{Lin:2014aa} to $A$ and $C=C(X)$, noting that, since $K_*(C(X))$ is assumed to be finitely generated and torsion free, any appearance of $KL$ in this theorem can be replaced by $K_*$. (We could alternatively appeal to the classification of maps into sequence algebras presented in \cite{Carrion:wz}, which also allows us to avoid mention of the UCT in the statement of Theorem~\ref{thm:bauer}.)
 
Let $\varepsilon>0$. The theorem provides us with $\delta\in(0,1)$ and a finite set $G\subseteq C(X)_+^1$ such that unital $^*$-monomorphisms $C(X)\to A$ that induce the same homomorphism $K_*(C(X))\to K_*(A)$, and tracially agree on $G$ up to $\delta$, are approximately unitarily conjugate on $\lip(X)$ up to $\frac{\varepsilon}{2k_X}\le\frac{\varepsilon}{2}$. 

\underline{Step 1}: perturbation. Since $F=(\varphi\cup\psi)(G)$ and $K=\partial_e(T(A))$ are compact, there is an open cover $\{U_i\}_{i=1}^m$ of $K$ such that
\begin{enumerate}[(i)]
\item \begin{equation}\label{eqn:small1} \max\limits_{1\le i\le m} \sup\limits_{\tau,\tau'\in U_i} \sup\limits_{a\in F} |\tau(a)-\tau'(a)| < \frac{\delta}{2};\end{equation}
\item $\bigcap\limits_{i\in I} U_i=\emptyset$ for any index set $I\subseteq\{1,\ldots,m\}$ of size $>d+1$ (where $d=\dim(K)$).
\end{enumerate}
Let $\{f_i\}_{i=1}^m$ be a partition of unity on $K$, and $\{\tau_i\}_{i=1}^m$ traces in $K$, with $\tau_i\in f_i^{-1}(\{1\})\subseteq\supp(f_i)\subseteq U_i$ for $1\le i \le m$. For each $i$, find $\mu_i,\nu_i\in\mathcal{M}_g(X)$ with
\begin{equation} \label{eqn:small2}
\sup_{f\in G}\max\{|\tau_{\mu_i}(f)-\tau_i(\varphi(f))|, |\tau_{\nu_i}(f)-\tau_i(\psi(f))|\} < \frac{\delta}{2}.
\end{equation}
By \cite[Theorem \rm{II}.3.12]{Alfsen:1971hl}, we may extend the functions $f_i$ to continuous affine maps $T(A)\to[0,1]$, and then define $\lambda_1,\lambda_2\colon T(A)\to T(C(X))$ by
\begin{equation} \label{eqn:pou}
\lambda_1(\tau) = \sum_{i=1}^m f_i(\tau)\tau_{\mu_i}, \quad \lambda_2(\tau) = \sum_{i=1}^m f_i(\tau)\tau_{\nu_i}.
\end{equation}
By \cite[Theorem 2.6]{Matui:2011uq}, there are unital $^*$-monomorphisms $\varphi',\psi'\colon C(X)\to A$ such that $KL(\varphi')=KL(\varphi)$ and $\tau\circ\varphi'=\lambda_1(\tau)$ for every $\tau\in T(A)$, and similarly for $\psi'$, $\psi$ and $\lambda_2$. For $f\in G$ and $\tau\in K$ we have by (\ref{eqn:small1}) and (\ref{eqn:small2}) that
\begin{align*}
|\tau(\varphi'(f)) - \tau(\varphi(f))| &= \left|\sum_{i=1}^m f_i(\tau)\tau_{\mu_i}(f) - \sum_{i=1}^m f_i(\tau)\tau(\varphi(f))\right|\\
&\le \sum_{\{i\mid \tau\in U_i\}} (|\tau_{\mu_i}(f)-\tau_i(\varphi(f))| + |\tau_i(\varphi(f))-\tau(\varphi(f))|)f_i(\tau)\\
&< \delta,
\end{align*}
so $\varphi'$ and $\varphi$ are unitarily conjugate on $\lip(X)$ up to $\varepsilon$ (and similarly for $\psi'$ and $\psi$). In conclusion, we may assume that $\varphi^*, \psi^*\colon T(A)\to T(C(X))$ are given by (\ref{eqn:pou}); in particular, $\mu_{\varphi^*\tau}, \mu_{\psi^*\tau}\in\mathcal{M}_g(X)$ for every $\tau\in K$.

\underline{Step 2}: diagonalization. Let $q_1,\dots,q_l\in M_k(C(X))$ be projections that generate $K_0(C(X))$, and for each $j$, let $q'_j=\psi(q_j)\in M_k(A)$ and let $r_j$ be the common value of $\tau(q_j)$ for all (non-normalized) traces $\tau$ on $M_k(C(X))$ (that is, $r_j=\rank(q_j)$). Set $r=\min_{1\le j\le l}r_j$, and let $\gamma>0$ be small enough such that
\[
\gamma < \min\left\{ \frac{\delta}{4(d+1)}, \frac{r}{3k}-\frac{r^2}{(3k)^2}, \frac{r}{4km^3} \right\}.
\]
By \cite[Lemma 3.16]{Bosa:aa} and \cite[\S1.3(ii)]{Castillejos:2021wm}, together with central surjectivity \cite[Lemma 1.8]{Castillejos:2021wm} and the fact that $A$ has real rank zero, there exist projections $p_1,\ldots,p_m$ in $A$ such that $p_1+\cdots+p_m=1$,
\begin{equation} \label{eqn:proj}
\max_{1\le i\le m}\sup_{\tau\in K}|\tau(p_i)-f_i(\tau)| < \gamma
\end{equation}
and
\begin{equation} \label{eqn:commute}
\max_{1\le i\le m}\max_{1\le j \le l}\|[p'_i,q'_j]\|< \gamma
\end{equation}
(where $p'_i=p_i\otimes1_k$). In particular, for every $1\le i \le m$ and $1\le j \le l$, by (\ref{eqn:commute})  there is a projection $q_{j,i}\in p_i'M_k(A)p_i'\subseteq M_k(p_iAp_i)$ with $\|p_i'q_j'p_i'-q_{i,j}\|<\frac{r}{3k}$. Then, using (\ref{eqn:commute}) again,
\begin{align*}
\left\|q'_j - \sum_{i=1}^mq_{i,j}\right\| &= \left\|\left(\sum_{i=1}^mp'_i\right)q'_j\left(\sum_{i=1}^mp'_i\right) - \sum_{i=1}^mq_{i,j}\right\|\\
&< \left\|\sum_{i=1}^mp'_iq'_jp'_i - \sum_{i=1}^mq_{i,j}\right\| +m^2 \cdot \frac{r}{4km^3}\\
&< \max_{1\le i\le m}\|p'_iq'_jp'_i - q_{i,j}\| + \frac{1}{2m}\\
&< \frac{r}{3k} + \frac{1}{2m}\\
&< 1.
\end{align*}
This implies that for each $j$, $[\psi(q_j)] = K_0(\iota_1)[q_{1,j}]+\dots+K_0(\iota_m)[q_{m,j}]$ in $K_0(A)$, where $\iota_i\colon p_iAp_i\hookrightarrow A$ are the inclusion maps. Moreover, we can choose each $q_{j,i}$ to be nonzero. Otherwise, if $q_{j,i}=0$ for some $i$ and $j$, then
\begin{align*}
r_j = \tau_i\circ\psi(q_j) = \tau_i(q'_{j}) & < \tau_i\left( \sum_{t\neq i}p'_tq'_jp'_t \right) + k\left( m^2\gamma + \|p_i'q_j'p_i'\| \right)\\
&< k\left( (1-\tau_i(p_i)) +\frac{m^2r}{4km^3} + \frac{r}{3k} \right)\\
&< \frac{r}{3} + \frac{r}{3} +\frac{r}{3} = r \le r_j.
\end{align*}
\sloppy
Another application of \cite[Theorem 2.6]{Matui:2011uq} gives unital $^*$-monomorphisms $\psi_i\colon C(X)\to p_iAp_i$ such that
\begin{items}
\item $K_1(\psi_1)=K_1(\psi)$ (under the isomorphism $K_1(\iota_1)\colon K_1(p_1Ap_1)\to K_1(A)$ induced by inclusion) and $K_1(\psi_i)=0$ for $i>1$;
\item $\psi_i^*$ maps every tracial state on $p_iAp_i$ to $\tau_{\nu_i}$;
\item $K_0(\psi_i)([q_j])=[q_{i,j}]$ for every $j$, so that $K_0(\iota_1\circ\psi_1) + \dots + K_0(\iota_m\circ\psi_m) = K_0(\psi)$.
\end{items}
\fussy
Let $\psi'=\iota_1\circ\psi_1+\cdots+\iota_m\circ\psi_m$. Then $KL(\psi')=KL(\psi)$, and for $g\in G$ and $\tau\in K$ we have by (\ref{eqn:pou}) and (\ref{eqn:proj}) that
\begin{align*}
|\tau(\psi'(g)) - \tau(\psi(g))| & = \left|\sum_{i=1}^m\tau(p_i)\tau_{\nu_i}(g) - \sum_{i=1}^m f_i(\tau)\tau_{\nu_i}(g)\right|\\
&\le \sum_{\{i\mid \tau\notin U_i\}}\tau(p_i) + \sum_{\{i\mid \tau\in U_i\}}|\tau(p_i)-f_i(\tau)|\tau_{\nu_i}(g)\\
&< \left(1-\sum_{\{i\mid \tau\in U_i\}}\tau(p_i)\right)+\frac{\delta}{4}\\
&\le \sum_{\{i\mid \tau\in U_i\}}|f_i(\tau)-\tau(p_i)|+\frac{\delta}{4}\\
&< \frac{\delta}{2},
\end{align*}
so $\psi'$ and $\psi$ are unitarily conjugate on $\lip(X)$ up to $\varepsilon$. In conclusion, we may assume that $\psi$ has the diagonal factorization (\ref{eqn:diagonal}).

\sloppy
\underline{Step 3}: transportation. By assumption, there exist homeomorphisms $h_i\colon X\to X$ homotopic to the identity such that
\begin{equation} \label{eqn:transport}
\max_{1\le i\le m}\sup_{x\in X}d(h_i(x),x) < k_XW_\infty(\mu_i,\nu_i)+\frac{\varepsilon}{2}
\end{equation}
and
\begin{equation} \label{eqn:weak}
\max_{1\le i\le m}\sup_{g\in G}\left|\tau_{(h_i)_*\nu_i}(g)-\tau_{\mu_i}(g)\right| < \frac{\delta}{2}.
\end{equation}

Let $\varphi'=\bigoplus_{i=1}^m(\iota_i\circ\psi_i\circ h_i^*)\colon C(X)\to A$. Then $KL(\varphi')=KL(\psi)=KL(\varphi)$, and using (\ref{eqn:weak}) and the same estimate as in Step 2, one checks that
\[
|\tau(\varphi'(g)) - \tau(\varphi(g))| < \delta
\]
for every $g\in G$ and $\tau\in K$, so that $d_\mathcal{U}(\varphi',\varphi)<\frac{\varepsilon}{2k_X}$. Hence, by \cite[Proposition 3.4]{Jacelon:2021wa}, $W_\infty(\varphi',\varphi)<\frac{\varepsilon}{2k_X}$ too. By the same string of inequalities that concludes the proof of \cite[Theorem 4.11]{Jacelon:2021wa}, except using (\ref{eqn:transport}) instead of \cite[(4.3)]{Jacelon:2021wa}, we obtain
\begin{align*}
d_\mathcal{U}(\varphi,\psi) &\le d_\mathcal{U}(\varphi',\psi) + \frac{\varepsilon}{2}\\
&\le k_XW_\infty(\varphi,\psi) + \varepsilon\\
&\le k_XW_\infty(\varphi',\psi) + \frac{3\varepsilon}{2}\\
&\le k_Xd_\mathcal{U}(\varphi,\psi) + 2\varepsilon,
\end{align*}
and therefore the desired inequality.
\fussy
\end{proof}

\begin{remark} \label{remark:bauercomments}
\begin{enumerate}[(i)]
\item Theorem~\ref{thm:bauer} applies to, for example, $X=\mathrm{U}(n)$, $X=\mathrm{SU}(n)$ and $X=\mathrm{Sp}(n)$ (see Remark~\ref{remark:lie}). In these cases, $k_X$=1, so we obtain equality of $d_\mathcal{U}$ and $W_\infty$.
\item There is actually no need to restrict to \emph{real-valued} Lipschitz functions in the definition (\ref{eqn:unitary}) of $d_\mathcal{U}$ and the proof of Theorem~\ref{thm:bauer}. (Indeed, allowing for arbitrary Lipschitz functions $X\to\mathbb{C}$ is how one deduces \cite[Corollary 4.12]{Jacelon:2021wa} from \cite[Theorem 4.11]{Jacelon:2021wa}.) However, this restriction will allow us to obtain a version of the theorem for $X=[0,1]$ and $p\in[1,\infty)$ (see Theorem~\ref{thm:interval}). We will also see that for $X=[0,1]$ and $p=\infty$, all restrictions on the trace space can be removed.
\end{enumerate}
\end{remark}

\section{One-dimensional NCCW complexes} \label{section:onedim}

\begin{definition}
A \emph{one-dimensional NCCW complex} is a pullback of the form
\[
A = \{(f,g) \in C([0,1],F)\oplus E \mid f(0) = \alpha_0(g), f(1) = \alpha_1(g)\}
\]
for finite dimensional $\cs$-algebras $E$ and $F$ and $^*$-homomorphisms $\alpha_0,\alpha_1:E\to F$ (which we will always assume provide an injective map $\alpha=\alpha_0\oplus\alpha_1\colon E\to F\oplus F$). Examples include \emph{dimension drop algebras} (see \S\ref{subsection:dimdrop}) and \emph{Razak blocks} (see \S\ref{subsection:razak}).

The components of $E=\bigoplus_{j=1}^kM_{i_j}$ correspond to \emph{points at infinity}, which we will denote by $\{\infty_i\}_{i=1}^k$.
\end{definition}

Every one-dimensional NCCW complex $A$ is semiprojective (see \cite[Theorem 6.2.2]{Eilers:1998yu}), which implies for us that a $^*$-homomorphism from $A$ into an inductive limit can be approximated by a $^*$-homomorphism into a finite stage (see \cite[Lemma 3.7]{Loring:1993kq}). We will use this property in the proofs of Theorem~\ref{thm:jiangsu} and Theorem~\ref{thm:razak}. 

\begin{definition} \label{definition:diagonal}
\sloppy
Let $A\subseteq C([0,1],M_n)$ and $B\subseteq C([0,1],M_m)$ be one-dimensional NCCW complexes. Following \cite{Jacelon:2021uc}, we call a $^*$-homomorphism $\varphi\colon A\to B$ \emph{diagonal} if $m=ln$ for some $l\in\mathbb{N}$, and there are continuous maps $\xi_1,\dots,\xi_l\colon [0,1]\to[0,1]$ and unitaries $u(t)\in M_m$, $t\in[0,1]$, such that  $\xi_i\leq\xi_{i+1}$ for all $i$, and 
\[
\varphi(f)(t) = u(t)\diag(f(\xi_1(t)),\ldots,f(\xi_l(t)))u(t)^* \text{ for all } f\in A,\,t\in [0,1].
\] 
The maps $\{\xi_i\}$ are said to be \emph{associated} to $\varphi$. We define the \emph{diagonal distance} between $\varphi$ and $\psi$, with associated maps $\{\xi_i^\varphi\}$ and $\{\xi_i^\psi\}$ respectively, to be
\begin{equation}
d_\partial(\varphi,\psi) = \sup_{t\in[0,1]}\max_{1\le i\le l}|\xi_i^\varphi(t)-\xi_i^\psi(t)|.
\end{equation}
\fussy
\end{definition}

Recall (cf.\ \cite[Definition 3.1]{Jacelon:2021wa}) that the Cuntz distance between normal elements $a$ and $b$ in a $\cs$-algebra $A$ is
\[
d_W(a,b) = \sup_{ U\subseteq[0,\infty)\text{ open }} \inf\{r>0 \mid f_U(a) \precsim f_{U_r}(b) \, , \, f_U(b) \precsim f_{U_r}(a)\},
\]
where $f_O$ denotes a(ny) positive continuous function with open support $O$, and $U_r$ is the $r$-neighbourhood of $U$. By Hall's Marriage Theorem, if $a$ and $b$ are normal matrices in $M_n$, then $d_W(a,b)$ is equal to the optimal matching distance between their eigenvalues.

\begin{definition} \label{definition:nice}
\sloppy
For the purposes of this section, let us say that a one-dimensional NCCW complex $A$ is \emph{nice} if
\begin{items}
\item $A$ is supported on one line, that is, $F$ is a single matrix algebra; and
\item  the spectrum of $A$ (under the hull-kernel topology) is Hausdorff, that is, $\alpha_0$ and $\alpha_1$ are each supported on a single component of $E$, or equivalently, the representations $\pi_t$ converge as $t\to0$ or $t\to1$ to single points at infinity.
\end{items}
Given a nice one-dimensional NCCW complex $A\subseteq C([0,1],M_n)$, let $\lip(A)$ be the compact set of contractive self-adjoint functions
\[
\lip(A) = \{f \in A^1_{sa} \mid \forall\,x,y\in [0,1]\, (\|f(x)-f(y)\| \le |x-y|)\}.
\]
We will measure the unitary and Cuntz distances between $^*$-monomorphisms $\varphi$ and $\psi $ from $A$ to a $\cs$-algebra $B$ relative to what we will call a \emph{Lipschitz family}: a subset
\begin{equation} \label{eqn:lipfamily}
\mathcal{L} = u\left(\lip(u^*Au)\right)u^* \subseteq A,
\end{equation}
for $u\in C([0,1],M_n)$ a continuous path of unitaries between permutation matrices $u(0)$ and $u(1)$. We define these distances as
\begin{equation} \label{eqn:du}
d_\mathcal{U}(\varphi,\psi) = \sup_{F\subseteq\mathcal{L}\,\text{finite }}\inf_{v\in\mathcal{U}(\tilde B)}\sup_{f\in F}\|v\varphi(f)v^*-\psi(f)\|
\end{equation}
and
\begin{equation} \label{eqn:dw}
d_W(\varphi,\psi) = \sup_{f\in\mathcal{L}} d_W(\varphi(f),\psi(f))
\end{equation}
respectively. The tracial distances $W_p(\varphi,\psi)$, for $1\le p\le \infty$, are defined to extend the equivalent definition in the commutative case provided by Proposition~\ref{prop:lipschitz}, that is, we equip $T(A)$ with the metric
\begin{equation} \label{eqn:metric}
W_p(\tau_1,\tau_2) = \sup_{f\in\mathcal{L}} W_p(\mu_{\tau_1\circ\iota_f},\mu_{\tau_2\circ\iota_f}),
\end{equation}
where $\iota_f\colon\sigma(f)\to A_{n,k}$ denotes the Gelfand $^*$-homomorphism, and define
\begin{equation} \label{eqn:wpp}
W_p(\varphi,\psi) = \sup_{\tau\in T(B)}W_p(\tau\circ\varphi,\tau\circ\psi).
\end{equation}
\fussy
\end{definition}

\begin{remark} \label{remark:lipschitz}
\sloppy
\begin{enumerate}[(i)]
\item \label{lip1} The role of the unitary $u$ in Definition~\ref{definition:nice} is to accommodate the one that appears in Definition~\ref{definition:diagonal}. We would like a diagonal $^*$-homomorphism whose associated maps are $1$-Lipschitz to have the ability to map one Lipschitz family into another (see Definition~\ref{definition:limit}).
\item \label{lip2} A Lipschitz family $\mathcal{L}$ associated to $A$ generates its Cuntz semigroup, so is a reasonable reference family for distance measurement. (In fact, by a suitable version of Stone-Weierstrass such as \cite[Corollary 1]{Prolla:1994td}, $\bigcup_{n\in\mathbb{N}}n\mathcal{L}$ is norm dense in $A$. Since $\mathcal{L}$ is closed under cutdowns $f \mapsto (f-\varepsilon)_+$, it can then be deduced that every Cuntz class in $\cu(A)$ is the supremum of an increasing sequence from $\mathcal{L}$.) In particular, if $A$ is a Razak block or is unital and $K_1(A)=0$, then by the results of \cite{Robert:2010qy}, for any stable rank one $\cs$-algebra $B$ and any $^*$-homomorphisms $\varphi,\psi\colon A\to B$, if $d_W(\varphi,\psi)=0$, then $\cu(\varphi)=\cu(\psi)$, and so $\varphi$ and $\psi$ are approximately unitarily equivalent.
\item \label{lip3} The definition of the metrics $W_p$ is motivated by \cite{Rieffel:1999aa}: for $A=C([0,1],M_n)$ and $p=1$, (\ref{eqn:metric}) is equivalent to
\[
W_1(\tau_1,\tau_2) = \sup_{a\in\mathcal{L}}|\tau_1(a)-\tau_2(a)|,
\]
which one should compare with the metric $\rho_L$ found in \cite[\S2]{Rieffel:1999aa}. In general, by the Kantorovich--Rubinstein Theorem (see Remark~\ref{remark:representations}),
\[
\sup_{a\in\mathcal{L}}|\tau_1(a)-\tau_2(a)| \le W_1(\tau_1,\tau_2) = \sup_{a\in\mathcal{L}, h\in\lip(C_0(0,1])}|\tau_1(h(a))-\tau_2(h(a))|,
\]
which implies that $W_1$ induces the $w^*$-topology on $T(A)$. Consequently, so does $W_p$ for every $p\in[1,\infty)$, and by Proposition~\ref{prop:bottleneck}, so does $W_\infty$ when restricted to the set of \emph{faithful} traces on $A$, provided that $A$ has no nontrivial projections. Moreover, for every $a\in\mathcal{L}$, the function $\hat{a}\colon T(A)\to\mathbb{R}$, $\tau\mapsto\tau(a)$ is in $\lip(T(A),W_1)$, hence in $\lip(T(A),W_p)$ for every $p\in[1,\infty]$.
\item \label{lip4} If $A$ has no nontrivial projections and $B$ is a simple, exact $\cs$-algebra with strict comparison, then $d_W(\varphi,\psi)=W_\infty(\varphi,\psi)$ (see \cite[Lemma 3.3 (iii),(iv)]{Jacelon:2021wa}).
\item \label{lip5} If $A$ is a nice one-dimensional NCCW complex such as $C([0,1],M_n)$ or a dimension drop algebra
\[
Z_{p,q} = \{f\in C([0,1],M_{p}\otimes M_q)\mid f(0) \in M_p\otimes 1_q,\,f(1) \in 1_p\otimes M_q \}
\]
or a Razak block
\begin{eqnarray*}
A_{n,k}=\{f\in C([0,1],M_{nk})\mid \exists a\in M_k(f(0)=&\diag(\underbrace{a}_\text{n}),\\ f(1)=&\diag(\underbrace{a}_\text{n-1},\underbrace{0}_k))\}
\end{eqnarray*}
we will benefit from the existence of an element $g_\mathcal{L}\in\mathcal{L}$ that has the property that, for every $l\in\mathbb{N}$ and points $s_1\le\dots \le s_l$ and $t_1\le\dots \le t_l$ in $[0,1]$,
\[
d_W(\diag(g_\mathcal{L}(s_1),\dots,g_\mathcal{L}(s_l)),\diag(g_\mathcal{L}(t_1),\dots,g_\mathcal{L}(t_l)))=\max_{1\le i\le l}|s_i-t_i|.
\]
In the first two example cases, we may take $g_\mathcal{L}(t)=\diag(\underbrace{t}_n)$, while for a Razak block $A_{n,k}$, this role is fulfilled by $g_\mathcal{L}(t)=u(t)u(1)^*\diag(\underbrace{1}_{(n-1)k},\underbrace{1-t}_k)u(1)u(t)^*$.
\end{enumerate}
\fussy
\end{remark}

\begin{proposition} \label{prop:diagonal}
\sloppy
Let $A\subseteq C([0,1],M_n)$ and $B\subseteq C([0,1],M_m)$ be nice one-dimensional NCCW complexes, and let $\varphi,\psi\colon A\to B$ be diagonal $^*$-monomorphisms. Then, relative to any Lipschitz family $\mathcal{L}\subseteq A$ containing an element $g_\mathcal{L}$ as above,
\[
d_\partial(\varphi,\psi) \le d_W(\varphi,\psi) \le d_\mathcal{U}(\varphi,\psi).
\]
\fussy
\end{proposition}

\begin{proof}
Let $\{\xi_i^\varphi\}_{1\le i\le l}$ and $\{\xi_i^\psi\}_{1\le i\le l}$ be the maps associated to $\varphi$ and $\psi$. Then,
\begin{align*}
d_\partial(\varphi,\psi) &= \sup_{t\in[0,1]}\max_{1\le i\le l}|\xi_i^\varphi(t)-\xi_i^\psi(t)|\\
&= \sup_{t\in[0,1]}d_W(\varphi(g_\mathcal{L})(t),\psi(g_\mathcal{L})(t))\\
&\le \sup_{f\in\mathcal{L}}\sup_{t\in[0,1]}d_W(\varphi(f)(t),\psi(f)(t))\\
&\le \sup_{f\in\mathcal{L}}d_W(\varphi(f),\psi(f)) &&(\:=d_W(\varphi,\psi))\\
&\le \sup_{f\in\mathcal{L}}d_\mathcal{U}(\varphi(f),\psi(f)) &&\text{(by \cite[Lemma 1]{Robert:2010rz})}\\
&\le d_\mathcal{U}(\varphi,\psi) &&\text{(by definition)}. \qedhere
\end{align*}
\end{proof}

It should be noted that, while \cite[Lemma 1]{Robert:2010rz} is stated for positive elements, the same proof works for elements that are self adjoint, provided that one allows for negative values of $t$. In addition, while the definition of $d_W$ that appears in \cite{Robert:2010rz} involves only those open sets $U$ of the form $(r,\infty)$, the inequality $d_W\le d_\mathcal{U}$ that we used in the proof of Proposition~\ref{prop:diagonal} remains true for the \emph{a priori} larger distance of Definition~\ref{definition:nice}. This follows from \cite[Lemma 1]{Cheong:2015aa}, together with \cite[Theorem 4.6]{Toms:2009uq}.

\begin{definition} \label{definition:limit}
By a \emph{$1$-Lipschitz system} we mean a direct system $(A_i,\mathcal{L}_i,\varphi_i)_{i\in\mathbb{N}}$, where each $A_i$ is a nice one-dimensional NCCW complex with Lipschitz family $\mathcal{L}_i\subseteq A_i$, and each $\varphi_i\colon A_i\to A_{i+1}$ is a diagonal $^*$-monomorphism with $\varphi_i(\mathcal{L}_i)\subseteq\mathcal{L}_{i+1}$. In this case, we will use
\begin{equation}
\mathcal{L}(A)=\bigcup_{i\in\mathbb{N}}\varphi_{i,\infty}(\mathcal{L}_i)
\end{equation}
to measure the distances (\ref{eqn:du}), (\ref{eqn:dw}) and (\ref{eqn:wpp}) between $^*$-homomorphisms $\varphi,\psi$ from the limit $A=\varinjlim(A_i,\varphi_i)$ to a $\cs$-algebra $B$ (of stable rank one).
\end{definition}

\begin{remark} \label{remark:range}
The inductive limit models constructed in \cite{Jiang:1999hb} and \cite{Tsang:2005fj} to exhaust the range of the tracial invariant rely on the Krein--Milman type theorem of \cite{Li:1999aa} (which is based on \cite{Thomsen:1994qy}). The connecting maps of these constructions are in general not $1$-Lipschitz. Nonetheless, there are interesting examples covered by the results of this section. Some of these are described in \S\ref{section:examples}.
\end{remark}

\subsection{Dimension drop algebras} \label{subsection:dimdrop}

By a \emph{dimension drop algebra}, we mean a nice one-dimensional NCCW complex of the form
\[
Z_{p,q}=\{f\in C([0,1],M_{p}\otimes M_q)\mid f(0) \in M_p\otimes 1_q,\,f(1) \in 1_p\otimes M_q \},
\]
for some $p,q\in\mathbb{N}$. We say that $Z_{p,q}$ is \emph{prime} if $p$ and $q$ are coprime. These $\cs$-algebras are the building blocks of the Jiang--Su algebra $\js$ \cite{Jiang:1999hb}.

\begin{proposition} \label{prop:dimdrop}
Let $A=Z_{p,q}$ and $B=Z_{p',q'}$ be prime dimension drop algebras, and let $\varphi,\psi\colon A\to B$ be (necessarily unital) $^*$-monomorphisms. Then, relative to any Lipschitz family $\mathcal{L}\subseteq A$, $d_\mathcal{U}(\varphi,\psi) = d_\partial(\varphi,\psi) = d_W(\varphi,\psi)$.
\end{proposition}

\begin{proof}
By Proposition~\ref{prop:diagonal}, it suffices to show that $d_\mathcal{U}(\varphi,\psi) \le d_\partial(\varphi,\psi)$. A self contained demonstration of this inequality is provided by \cite{Masumoto:2017wx}. Here is a summary of the argument. Let $\varepsilon>0$. By \cite[Proposition 4.7]{Masumoto:2017wx}, $\varphi_1=\varphi$ and $\varphi_2=\psi$ may be assumed to be diagonal:
\begin{equation} \label{eqn:masumoto}
\varphi_i(f)=u_i\diag(f\circ\xi_1^i,\ldots,f\circ\xi_l^i)u_i^*, \: i=1,2,
\end{equation}
such that the associated unitaries $u_1$ and $u_2$ are continuous. The approximation there is stated for finite subsets, but compactness of $\mathcal{L}$ provides the required equicontinuity used for example in \cite[Proposition 3.5]{Masumoto:2017wx}. By \cite[Lemma 4.2]{Masumoto:2017wx}, there is a common unitary $v\in C([0,1],M_{p'q'})$ such that
\[
\psi_i \colon f\mapsto v\diag(f\circ\xi_1^i,\ldots,f\circ\xi_l^i)v^*, \: i=1,2,
\]
are $^*$-homomorphisms from $A$ to $B$. We may assume that $v$ is of the form $v(t)=x(t)\diag(\underbrace{u(t)^*}_l)$, where $u$ is the unitary associated to the Lipschitz family $\mathcal{L}$ (see Definition~\ref{definition:nice}), and $x(0)$, $x(1)$ are suitable permutation unitaries. Moreover, there are unitaries $y_i(0)\in M_{p'}\otimes1_{q'}$ and $y_i(1)\in 1_{p'}\otimes M_{q'}$, $i=1,2$, such that \[y_i(s)\varphi_i(f)(s)y_i(s)^* = \psi_i(f)(s) \text{ for all } f\in A,\, s\in\{0,1\}.\]

By \cite[Lemma 4.9]{Masumoto:2017wx}, there are unitaries $w_i\in C([0,1],M_{p'q'})$ such that
\[
\sup_{f\in\mathcal{L}}\|w_i\varphi(f)w_i^*-\psi_i(f)\| < \frac{\varepsilon}{2}, \: i=1,2.
\]
\sloppy
Connecting $w_i(s)^*y_i(s)$ to $1_{p'q'}$ via a path of unitaries commuting with $\varphi_i(A)$, we may in fact assume that $w_i(s)=y_i(s)$ for $s\in\{0,1\}$, that is, $w_1,w_2\in B$. Then, as in \cite[Proposition 4.10]{Masumoto:2017wx},
\begin{align*}
d_\mathcal{U}(\varphi,\psi) &\le \sup_{f\in\mathcal{L}}\|w_2\varphi_2(f)w_2^*-w_1\varphi_1(f)w_1^*\| \\
&\le \sup_{f\in\mathcal{L}}\|\psi_2(f)-\psi_1(f)\| + \varepsilon \\
&\le \sup_{f\in\lip(C([0,1],M_{pq}))} \sup_{t\in[0,1]} \max_{1\le j\le l} \|f(\xi_j^2(t))-f(\xi_j^1(t))\| + \varepsilon \\
&\le \sup_{t\in[0,1]} \max_{1\le j\le l} |\xi_j^2(t)-\xi_j^1(t)| + \varepsilon \\
&= d_\partial(\varphi,\psi) + \varepsilon. \qedhere
\end{align*}
\fussy
\end{proof}

Recall that \emph{pureness} is a regularity property of the Cuntz semigroup that is satisfied by, for example, simple, separable, $\js$-stable $\cs$-algebras (see \cite[\S3]{Winter:2012pi}).

\begin{proposition} \label{prop:model}
Let $B$ be an infinite-dimensional, simple, separable, unital, exact, pure $\cs$-algebra of stable rank one. Then, there is a simple inductive limit $C$ of prime dimension drop algebras and a unital embedding $\iota\colon C \to B$ that induces an affine homeomorphism $T(B)\to T(C)$.
\end{proposition}

\begin{proof}
This follows from \cite[Theorem 4.5]{Jiang:1999hb} (which says that there exists $C$ whose tracial simplex is affinely homeomorphic to that of $B$) and the results of \cite{Robert:2010qy} (which in particular imply that the natural embedding of $\cu(C)$ into $\cu(B)$ lifts to an embedding of $C$ into $B$).
\end{proof}

\begin{theorem} \label{thm:jiangsu}
Let $A$ be the direct limit of a $1$-Lipschitz system $(Z_{p_i,q_i},\mathcal{L}_i,\varphi_i)$, let $B$ be an infinite-dimensional, simple, separable, unital, exact, pure $\cs$-algebra of stable rank one, and let $\varphi,\psi\colon A\to B$ be unital $^*$-monomorphisms. Then, relative to $\mathcal{L}(A)=\bigcup_{i\in\mathbb{N}}\varphi_{i,\infty}(\mathcal{L}_i)$,
\[
d_\mathcal{U}(\varphi,\psi) = d_W(\varphi,\psi) = W_\infty(\varphi,\psi).
\]
\end{theorem}

\begin{proof}
We proceed as in \cite[Theorem 4.1]{Jacelon:2014aa}, assuming without loss of generality that $A=Z_{p,q}$. At the level of the Cuntz semigroup, $\varphi$ and $\psi$ factor through the limit $C$ of Proposition~\ref{prop:model}. By \cite{Robert:2010qy}, this factorization is at the $^*$-homomorphism level (up to approximate unitary equivalence). By semiprojectivity, we can then assume that $\varphi$ and $\psi$ in fact map into some $Z_{p',q'}$. The theorem now follows from Proposition~\ref{prop:dimdrop}.
\end{proof}

\subsection{The interval} \label{subsection:interval}

\begin{proposition} \label{prop:interval}
Let $\varphi,\psi$ be unital $^*$-monomorphisms from $C([0,1])$ to a nice one-dimensional NCCW complex $B$. Then, relative to $\mathcal{L}=\lip([0,1])$,
\begin{equation}
d_\mathcal{U}(\varphi,\psi) = d_\partial(\varphi,\psi) = d_W(\varphi,\psi).
\end{equation}
\end{proposition}

\begin{proof}
This follows from Proposition~\ref{prop:diagonal} and \cite[Lemma 4.1, Proposition 5.1]{Jacelon:2021uc}.
\end{proof}

\begin{theorem} \label{thm:interval}
\sloppy
Let $A$ be a simple, separable, unital, nuclear, finite, $\js$-stable $\cs$-algebra, let $\varphi,\psi\colon C([0,1])\to A$ be unital $^*$-monomorphisms, and let $p\in[1,\infty]$. If either $p=\infty$ or the real rank of $A$ is zero and the extreme boundary $\partial_e(T(A))$ of the tracial state space of $A$ is compact and of finite Lebesgue covering dimension, then, relative to $\lip([0,1])$,
\[
d_{\mathcal{U},p}(\varphi,\psi) = W_p(\varphi,\psi).
\]
\fussy
\end{theorem}

\begin{proof}
First suppose that the real rank of $A$ is zero and that $\partial_e(T(A))$ is compact and finite dimensional. The proof in this case is the same as that of Theorem~\ref{thm:bauer}, with suitably modified justification for the last string of inequalities. Namely, we use Proposition~\ref{prop:lipschitz} instead of \cite[Proposition 3.4]{Jacelon:2021wa} and Corollary~\ref{corollary:interval} instead of \cite[Corollary 3.6]{Jacelon:2021wa}, and the behaviour of the transport maps $h_i$ is governed by Proposition~\ref{prop:wasserstein}. Note that $d_{\mathcal{U},p}(\varphi',\varphi)\le d_{\mathcal{U}}(\varphi',\varphi)$, so the initial and final inequalities in the string hold with $d_{\mathcal{U},p}$ in place of $d_{\mathcal{U}}$. Everything else in the argument is unchanged.

The case $p=\infty$, without restriction on $T(A)$, follows from Proposition~\ref{prop:model} and Proposition~\ref{prop:interval}, just as in the proof of Theorem~\ref{thm:jiangsu}.
\end{proof}

\subsection{Razak blocks} \label{subsection:razak}

Recall that a \emph{Razak block} is a $\cs$-algebra of the form
\begin{eqnarray*}
A_{n,k}=\{f\in C([0,1],M_{nk})\mid \exists a\in M_k(f(0)=&\diag(\underbrace{a}_\text{n}),\\ f(1)=&\diag(\underbrace{a}_\text{n-1},\underbrace{0}_k))\}
\end{eqnarray*}
for some $n,k\in\mathbb{N}$ (so that $A_{n,k} \cong A_{n,1}\otimes M_k$).

For every trace $\tau$ on $A_{n,k}$, there is a unique Borel probability measure $\mu=\mu_\tau$ on $[0,1)$ such that $\tau(f)=\int_0^1 \tr_{nk}(f) d\mu$, where $\tr_N$ denotes the tracial state on $M_N$. We write $T_{g}(A_{n,k})$ for those traces corresponding to faithful, diffuse probability measures $\mu$, and we write $\lambda$ for the trace corresponding to the Lebesgue measure.

There is a unique simple inductive limit of Razak blocks $\mathcal{W}=\varinjlim(A_i,\alpha_i)$ that has a unique bounded trace (see \cite{Jacelon:2010fj}). We will use the `generic' construction of $\mathcal{W}$ in \cite[Definition 3.10]{Jacelon:2021uc}.

The proof of the following is very similar to that of Proposition~\ref{prop:dimdrop}.

\begin{proposition} \label{prop:razak}
Let $A=A_{n,k}$ and $B=A_{n',k'}$ be Razak blocks, and let $\varphi,\psi\colon A\to B$ be diagonal $^*$-monomorphisms. Then, relative to any Lipschitz family $\mathcal{L}\subseteq A$,
\[
d_\mathcal{U}(\varphi,\psi) = d_\partial(\varphi,\psi) = d_W(\varphi,\psi).
\]
\end{proposition}

\begin{proof}
Once again, we summarize the justification of $d_\mathcal{U}(\varphi,\psi) \le d_\partial(\varphi,\psi)$. Just as for dimension drop algebras, one first perturbs $\varphi_1=\varphi$ and $\varphi_2=\psi$ so that the unitaries $u_1$ and $u_2$ in their diagonal descriptions as
\[
\varphi_i(f)=u_i\diag(f\circ\xi_1^i,\ldots,f\circ\xi_l^i)u_i^*, \: i=1,2,
\]
are continuous. The argument for this is essentially the same as \cite[Proposition 3.5]{Masumoto:2017wx}. Next, one shows that there is a common unitary $v\in C([0,1],M_{p'q'})$ such that
\[
\psi_i \colon f\mapsto v\diag(f\circ\xi_1^i,\ldots,f\circ\xi_l^i)v^*, \: i=1,2,
\]
are $^*$-homomorphisms from $A$ to $B$. The argument is by counting eigenvalues, similar to what is done in \cite[Lemma 4.2]{Masumoto:2017wx}. The key point is that, with
\[
n_{s,t}^{(i)} = |\{j \mid \xi_j^i(s)=t\}|,
\]
it is easy to see that $\varphi_1$ and $\varphi_2$ satisfy
\[
n_{0,t}^{(i)} \equiv 0 \mod n' \quad,\quad n_{1,t}^{(i)} \equiv 0 \mod n'-1 \quad, \quad i=1,2,
\]
if $t\in(0,1)$, and from this it is not hard to show that
\[
n_{0,t}^{(1)} \equiv n_{0,t}^{(2)} \mod n' \quad,\quad n_{1,t}^{(1)} \equiv n_{1,t}^{(2)} \mod n'-1 
\]
if $t=0$ or $t=1$. Using this, one can find suitable permutation unitaries $v(0)$ and $v(1)$, and take $v$ as a unitary path between the two. The rest of the proof is exactly the same as that of Proposition~\ref{prop:dimdrop}
\end{proof}

\begin{theorem} \label{thm:razak}
Let $A$ be the direct limit of a $1$-Lipschitz system $(A_{n_i,k_i},\mathcal{L}_i,\varphi_i)$ with each $\varphi_i$ nondegenerate, let $B$ be an infinite-dimensional, algebraically simple, separable, exact, pure $\cs$-algebra of stable rank one, and let $\varphi,\psi\colon A\to B$ be nondegenerate $^*$-monomorphisms. Then, relative to $\mathcal{L}(A)=\bigcup_{i\in\mathbb{N}}\varphi_{i,\infty}(\mathcal{L}_i)$,
\[
d_{\mathcal{U}}(\varphi,\psi) = d_W(\varphi,\psi) = W_\infty(\varphi,\psi).
\]
\end{theorem}

\begin{proof}
It suffices to consider the case where $A=A_{n,k}$. Let $\varphi,\psi\colon A\to B$ be nondegenerate, and let $\varepsilon>0$. We again adapt the proof of \cite[Theorem 4.1]{Jacelon:2014aa}, and assume using  \cite{Robert:2010qy} that, up to approximate unitary equivalence, $\varphi$ and $\psi$ map into $\mathcal{W}\otimes C$, where $C=\varinjlim(C_i,\beta_i)$ is a unital AF algebra with $T(C)\cong T(B)$. By semiprojectivity, we may further assume that $\varphi$ and $\psi$ in fact map into some finite stage $A'=A_i \otimes (M_{N_1}\ \oplus \cdots \oplus M_{N_m})$. For $1\le j\le m$, let $\varphi_j$ and $\psi_j$ denote the components of $\varphi$ and $\psi$ into $A_i \otimes M_{N_j}$.

To apply Proposition~\ref{prop:razak}, we need each $\varphi_j$ and $\psi_j$ to be diagonal, which we can accomplish as follows. By \cite[Theorem 4.1]{Razak:2002kq} (see also \cite[Theorem 4.2]{Tsang:2003qy}), there are finite sets $G,H\subseteq A_+^1$ (independent of $A'$) such that $^*$-homomorphisms $\varphi,\varphi'$ from $A$ to $A'$ whose infima over $H$ on $T(A')$ are both larger than some $\delta>0$, and that tracially agree on $G$ up to $\delta$, are unitarily conjugate on $\mathcal{L}$ up to $\varepsilon$. By having moved far enough down the sequence $\varinjlim(A_i\otimes C_i,\alpha_i\otimes\beta_i)$, we can ensure that the infimum condition holds for $\varphi$ for some $\delta_1$. By universality of the construction of $\mathcal{W}$ in \cite{Jacelon:2021uc}, for all sufficiently large $i$, there are diagonal embeddings $\varphi_j'$ of $A_{n,k}$ into $A_i\otimes M_{N_j}$ for $1\le j\le m$ such that $\lambda\circ\varphi'_j\in T_g(A_{n,k})$ for every $j$ (as in \cite[Proposition 3.5]{Jacelon:2021uc}) and such that the infimum condition also holds for $\varphi'=\varphi'_1\oplus\dots\oplus\varphi'_m$ and some $\delta_2$. We set $\delta=\min\{\delta_1,\delta_2\}$, and again move far enough down the sequence (which does not affect the validity of the infimum condition) to ensure that
\begin{equation} \label{eqn:embedding}
\max_{1\le j\le m}\sup_{f\in G}\,\sup_{\tau,\tau'\in T(A_i \otimes M_{N_j})}|\tau(\varphi'_j(f))-\tau'(\varphi'_j(f))| < \frac{\delta}{4}
\end{equation}
and the same for $\varphi$. Since nondegenerate maps preserve tracial states, we may also assume that
\[
\max_{1\le j\le m}\frac{1}{\|\varphi_j^*\lambda\|}-1 < \frac{\delta}{4}.
\]
For each $1\le j\le m$, choose $\sigma_j\in T_{g}(A_{n,k})$ such that
\[
\sup_{f\in G}\left|\frac{\varphi_j^*\lambda}{\|\varphi_j^*\lambda\|}(f)-\sigma_j(f)\right| < \frac{\delta}{4}.
\]
By precomposing with automorphisms of $A_{n,k}$ induced by the transport maps of Proposition~\ref{prop:wasserstein}, we may assume that the maps $\varphi'_j$ in fact satisfy $\lambda\circ\varphi'_j=\sigma_j$. Then, for $1\le j\le m$, $f\in G$ and $\tau\in T(A_i \otimes M_{N_j})$,
\begin{align*}
|\tau(\varphi'_j(f))-\tau(\varphi_j(f))| &< |\lambda(\varphi'_j(f))-\lambda(\varphi_j(f))| + \frac{\delta}{2}\\
&= |\sigma_j(f)-\lambda(\varphi_j(f))| + \frac{\delta}{2}\\
&\le \left|\sigma_j(f) - \frac{\varphi_j^*\lambda}{\|\varphi_j^*\lambda\|}(f)\right| + \left(\frac{1}{\|\varphi_j^*\lambda\|}-1\right)\varphi_j^*\lambda(f) + \frac{\delta}{2}\\
&< \delta.
\end{align*}
Hence, $\varphi$ is unitarily conjugate to $\varphi'$ on $\mathcal{L}$ up to $\varepsilon$. We construct $\psi'\colon A \to A'$ similarly. The result $d_\mathcal{U}(\varphi,\psi)=d_W(\varphi,\psi)$ now follows from Proposition~\ref{prop:razak}. Finally, strict comparison of positive elements in $B$ implies that $d_W(\varphi,\psi)=W_\infty(\varphi,\psi)$ (see Remark~\ref{remark:lipschitz}).
\end{proof}

\subsection{Examples} \label{section:examples}

We conclude by providing some examples of direct limit domains to which the results of \S\ref{section:onedim} apply. Each example is the limit of a $1$-Lipschitz system in the sense of Definition~\ref{definition:limit}.

\subsubsection{$\js$ and $\mathcal{W}$} These algebras are constructed in \cite{Jiang:1999hb} and \cite{Jacelon:2010fj} respectively as limits of $1$-Lipschitz systems. However, our results are vacuous for these domains: by \cite{Robert:2010qy}, any two nondegenerate $^*$-homomorphisms from $\js$ or $\mathcal{W}$ to one of our specified codomains are approximately unitarily equivalent.

\subsubsection{Generalized prime dimension drop algebras} The algebras $Z_{\mathfrak{p},\mathfrak{q}}$ are defined in the same way as dimension drop algebras, except that $\mathfrak{p}$ and $\mathfrak{q}$ are allowed to be (coprime) supernatural numbers. It is explained in \cite[\S3]{Rordam:2009qy} how to view these as limits of $1$-Lipschitz systems of prime dimension drop algebras (and less trivially, how to view $\js$ as the limit of a stationary system associated to a trace collapsing endomorphism of a fixed $Z_{\mathfrak{p},\mathfrak{q}}$).

\subsubsection{Mapping tori} \sloppy This example is due to Leonel Robert. Fix a natural number $n\ge2$, and let $\gamma$ be an endomorphism of the UHF algebra $M_{(n-1)^\infty n^\infty}$ whose action on $K_0$ is multiplication by $\frac{n-1}{n}$. Viewing $M_{(n-1)^\infty n^\infty}$ as $\bigotimes M_{(n-1)n}$, we can think of $\gamma$ as the shift map $\gamma(a)=p\otimes a$ for elementary tensors $a$, where $p\in M_{n(n-1)}$ is a projection of rank $(n-1)^2$ (so that $\tr_n(p)=\frac{n-1}{n}$). Let $A$ be the mapping torus
\[
A = \{f\in C([0,1],M_{(n-1)^\infty n^\infty}) \mid f(1) = \gamma(f(0))\}.
\]
Then, $A$ is the limit of the $1$-Lipschitz system $A_{n,1} \to A_{n,(n-1)n} \to A_{n,(n-1)^2n^2} \to \dots$, where each connecting map is a suitable unitary conjugate of $f \mapsto \diag(\underbrace{f}_{(n-1)n})$.

The stabilized version of this construction (in which $\gamma$ is extended to an automorphism of $M_{(n-1)^\infty n^\infty}\otimes\mathbb{K}$) appears in \cite{Kishimoto:1996yu}. The algebra $\mathcal{W}\otimes\mathbb{K}$ arises as an inductive limit of such mapping tori (see \cite[Theorem 2.4]{Kishimoto:1996yu}; that one obtains $\mathcal{W}\otimes\mathbb{K}$ is a consequence of, for example, Razak's classification \cite{Razak:2002kq}). \fussy

\subsubsection{A simple direct limit with nontrivial trace space} Here is how one can obtain a simple, unital $\cs$-algebra $A$ with $\partial_e(T(A))\cong[0,1]$ as the limit of a $1$-Lipschitz system of prime dimension drop algebras (based on work of Leonel Robert). We adapt the construction of \cite[Proposition 2.5]{Jiang:1999hb} (using the same notation) so that at the $m$th step, the eigenvalue maps $\xi_i$ are all of the form $\xi_i(t)=t$ or $\xi_i=c_m$ for some constant $c_m\in[0,1]$ or $\xi_i(t)=\max\{c_m,t\}$. In other words, with suitable multiplicities, the eigenvalue pattern will look like this:

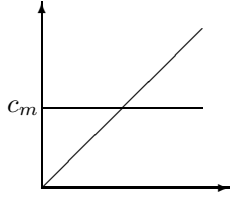
\begin{center}
\begin{picture}(70,70)
\put(0,0){\vector(1,0){70}}
\put(0,0){\vector(0,1){70}}
\put(0,0){\line(1,1){60}}
\put(0,30){\line(1,0){60}}
\put(-7,30){\makebox(0,0){$c_m$}}
\end{picture}
\end{center}

By allowing the constants $c_m$ to vary densely throughout $[0,1]$ as the sequence progresses (to ensure simplicity) and making the proportion of identity maps at each step sufficiently large, we will have obtained a system with the desired properties. `Sufficiently large' means large enough to secure an approximate intertwining
\[
\begin{tikzcd}
\aff(T(Z_{p_1,q_1})) \arrow[r] \arrow[d,equal] & \aff(T(Z_{p_2,q_2})) \arrow[r] \arrow[d,equal] & \dots \arrow[r] & \aff(T(A)) \arrow[d,dashed]\\
C_\mathbb{R}([0,1]) \arrow[r,"\id"] & C_\mathbb{R}([0,1]) \arrow[r,"\id"] & \dots \arrow[r,"\id"] & C_\mathbb{R}([0,1])
\end{tikzcd}
\]
(where we identify $C_\mathbb{R}([0,1])$ with $\aff(T(Z_{p_m,q_m}))$ via the embedding of $C([0,1])$ into the centre of $Z_{p_m,q_m}$, as in \cite[Lemma 2.4]{Jiang:1999hb}). If for each $m\in\mathbb{N}$ the proportion $\frac{1}{k}\#\{i\mid\xi_i=\id\}$ (where $k$ is the total number of eigenvalue maps) is at least $\frac{1}{m^2}$, so that the $m$th square commutes on the unit ball up to $\frac{1}{m^2}$, then by \cite[Lemma 3.4]{Thomsen:1994qy}, there is an induced isomorphism between $\aff(T(A))$ and $C_\mathbb{R}([0,1])$.

Given coprime $p_m=p<q=q_m$, let $k_0=p^n$ and $k_1=q^n$, where $n\ge2m$ is a sufficiently large natural number to ensure that $p^n>n^2q$ and $q^n>n^2p$. Set $p_{m+1}=k_0p_m=p^{n+1}$, $q_{m+1}=k_1q_m=q^{n+1}$ and $k=k_0k_1=p^nq^n$. In \cite[Proposition 2.5]{Jiang:1999hb}, the numbers of each type of eigenvalue map are determined by the values of $r_0=[k]_{q_{m+1}}$ and $r_1=[k]_{p_{m+1}}$. In particular, $k-r_0$ and $k-r_1$ are the numbers of occurrences of $f\left(\frac{1}{2}\right)$ in each block of $\varphi_m(f)(t)$ at $t=0$ and $t=1$ respectively. In our modified construction, we replace $\frac{1}{2}$ by $c_m$, and whittle these numbers down to the bare minimum, that is, we take $r_0=k-q_{m+1}=q^n(p^n-q)$ and $r_1=k-p_{m+1}=p^n(q^n-p)$. In other words, we demand exactly one $f(c_m)$ in each block. We set
\[
 \xi_i(t)=
 \begin{cases}
 t & 1\le i\le k-q_{m+1}\\
 c_m & k-q_{m+1}< i \le k-q_{m+1}+p_{m+1}\\
 \max\{c_m,t\} & k-q_{m+1}+p_{m+1} < i\le k.
 \end{cases}
\]
Since
\[
\frac{k-q_{m+1}}{k} = 1-\frac{q}{p^n} > 1-\frac{1}{n^2} > 1-\frac{1}{m^2},
\]
we have accomplished our task.

\section{Dynamics} \label{section:dynamics}

Finally, we adapt the construction of the last example of \S~\ref{section:examples} to prove Theorem~\ref{thm2}, restated here for the reader's convenience.

\begin{theorem} \label{thm:dynamics}
There exists a simple, separable, unital, nuclear, $\js$-stable, projectionless $\cs$-algebra $A$ that has trivial tracial pairing and satisfies the UCT, such that $\partial_e(T(A))\cong S^1$ and the set $\{a \mid \hat a|_{S^1} \text{ is Lipschitz}\}$ (with respect to the geodesic metric $d$ on $S^1$) of tracially Lipschitz elements is dense in $A_{sa}$. Moreover, there is a trace $\tau_0\in T(A)$ and a $\tau_0$-preserving endomorphism $\theta$ of $A$ such that, for every tracially Lipschitz $a$ with $\tau_0(a)=0$ and $\sigma_{\hat{a}}^2>0$ on $\partial_e(T(A))$, and almost every $\tau\in\partial_e(T(A))$, the sequence of weighted averages (\ref{eqn:weighted}) of the point masses
\[
\left\{\delta_{\frac{1}{\sqrt{k}}\tau(a+\theta(a)+\dots+\theta^{k-1}(a))}\right\}_{k=1}^n
\]
is $w^*$-convergent to the normal distribution $\mathcal{N}_{0,\sigma_{\hat{a}}^2}$.
\end{theorem}

\begin{proof}
Let $(p_m,q_m)_{m\in\nn}$ be the above sequence of coprime integers, and let $\{c_m\}_{m\in\nn}$ be a dense subset of $S^1\setminus\{-1,1\}$. For fixed $m\in\nn$, define
\[
A_m = \{f\in C(S^1,M_{p_m}\otimes M_{q_m}) \mid f(1) \in M_{p_m}\otimes 1_{q_m},\,f(-1) \in 1_{p_m}\otimes M_{q_m} \}.
\]
This $\cs$-algebra is a one-dimensional NCCW complex that is supported on two lines (the two semicircles joining $1$ and $-1$). It satisfies the UCT, is separable and nuclear, and has trivial tracial pairing (by connectedness of $S^1$, the trace of any projection is the value of its constant rank). In fact, $A_m$ has no nontrivial projections, and it is not hard to compute from the six-term exact sequence that
\[
(K_0(A_m),K_0(A_m)_+,[1_{A_m}],K_1(A_m)) = (\zz,\nn,1,\zz).
\]
Let $\varphi_m\colon A_m\to A_{m+1}$ be a diagonal $^*$-homomorphism (in the sense of Definition~\ref{definition:diagonal}, with $S^1$ in place of $[0,1]$ and with no ordering) whose eigenvalue maps are the functions $\xi_1,\dots,\xi_k\colon S^1\to S^1$, where $k=\frac{p_{m+1}q_{m+1}}{p_mq_m}$ and
\[
 \xi_i(z)=
 \begin{cases}
 z &  1\le i\le k-q_{m+1}\\
 c_m &  k-q_{m+1}< i \le k-q_{m+1}+p_{m+1}\\
 \gamma_m\left(\max\left\{\gamma_m^{-1}(c_m),\gamma_m^{-1}(\pi_m(z))\right\}\right) & k-q_{m+1}+p_{m+1} < i\le k.
 \end{cases}
\]
Here, $\gamma_m\colon[0,1]\to S^1$, $t\mapsto\exp(\pm i\pi t)$, parameterizes either the upper or lower semicircle, whichever contains $c_m$, and $\pi_m$ projects $S^1$ onto $\gamma_m([0,1])$.

Set $A=\varinjlim(A_m,\varphi_m)$. In the notation of \cite{Lin:2011ub}, $A$ is a simple, unital $\rm{A}\mathbb{T}\rm{D}$-algebra. It is $\mathcal{Z}$-stable (see \cite[Theorem 4.2]{Lin:2011ub} or \cite[Corollary 7.5]{Winter:2012pi}), hence classifiable. Just as in \S~\ref{section:examples}, $\partial_e(T(A))\cong \partial_e(T(A_m)) \cong S^1$. Let $(\mathcal{L}_m)_{m\in\nn}$ be the Lipschitz system (in the sense of Definition~\ref{definition:limit}) associated with
\[
\mathcal{L}_1 = \lip(A_1) = \{f \in (A_1^1)_{sa} \mid \forall\,x,y\in S^1 \, (\|f(x)-f(y)\| \le d(x,y))\}
\]
and $(\varphi_m)_{m\in\nn}$, and let $\mathcal{L}=\bigcup\limits_{m\in\nn}\varphi_{m,\infty}(\mathcal{L}_m)\subseteq A$.  Let $W_1$ be the metric on $X=\partial_e(T(A))$ induced by $\mathcal{L}$ in the sense described in Remark~\ref{remark:lipschitz}(\ref{lip3}), that is,
\[
W_1(\tau_1,\tau_2) = \sup_{a\in\mathcal{L}, h\in\lip(C_0(0,1])}|\tau_1(h(a))-\tau_2(h(a))|.
\]
By construction, $W_1$ is in fact the geodesic metric $d$ on $S^1\cong X$. For every $m\in\nn$, elements $f$ in the set
\[
K_m = \{f\in (A_m)_{sa} \mid \exists M>0 \, \forall\,x,y\in S^1 \, (\|f(x)-f(y)\| \le Md(x,y))\}
\]
have the property that $\widehat{\varphi_{m,\infty}(f)}$ is Lipschitz on $(X,W_1)$. By Stone--Weierstrass,  these elements are dense in $(A_m)_{sa}$, hence $\bigcup\limits_{m\in\nn}\varphi_{m,\infty}(K_m)$ is dense in $A_{sa}$. 

It remains to show the existence of the endomorphism $\theta$. By \cite[Theorem 6.3]{Lin:2011ub}, for any group homomorphism $\kappa_1\colon K_1(A)\to K_1(A)$ (the zero homomorphism will do) and continuous affine map $h\colon T(A)\to T(A)$, there is a unital $^*$-homomorphism $\theta\colon A\to A$ such that $K_1(\theta)=\kappa_1$ and $T(\theta)=h$. Note moreover that any continuous $h\colon X\to X$ can be extended to a continuous affine map $T(A)\to T(A)$ by pushing forward representing measures: for every $\tau\in T(A)$, there is a unique Borel probability measure $\mu_\tau$ supported on $X$ such that $f(\tau) = \int_{X}fd\mu_\tau$ for every $f\in\aff(T(A))$ (by definition of a metrizable Choquet simplex), and we define the extension by
\[
h(\tau)(a) = \int_X\hat a\circ h\, d\mu_\tau \quad\text{ for } a\in A_{sa}.
\]
As for the choice of $h\colon S^1\to S^1$, any strongly chaotic circle map known to satisfy the almost-sure CLT will do. Let us take the Pomeau--Manneville map with parameter $\alpha\in\left(0,1/2\right)$ (see for example \cite[\S3.5]{Chazottes:2015ti}), viewed as a map of the interval $[0,1]$ with its endpoints identified:
\[
 h(t)=
 \begin{cases}
 t+2^\alpha t^{1+\alpha} & \text{ if }\ 0\le t < \frac{1}{2}\\
 2t-1 & \text{ if }\ \frac{1}{2} \le t\le 1.
 \end{cases}
\]
There is a unique ergodic $h$-invariant probability measure $\mu_0$ which is equivalent to Lebesgue measure. The system $(S^1,h,\mu_0)$ satisfies the CLT \cite[Theorem 6]{Young:1999vu} and moreover the almost-sure CLT \cite[Theorem 18]{Chazottes:2015ti}: for any Lipschitz observable $f\colon S^1\to\mathbb{R}$, if $\int fd\mu_0=0$, and if the variance $\sigma_f^2$ (\ref{eqn:variance}) of $f$ is nonzero (which is the typical case), then for $\mu_0$-a.e.\ $t\in S^1$, the sequence of weighted averages
\begin{equation} \label{eqn:weighted}
T_n(t)=\frac{1}{D_n}\sum_{k=1}^n\frac{1}{k}\delta_{S_kf(t)/\sqrt{k}},
\end{equation}
where $S_kf(t)=\sum_{i=0}^{k-1}f(h^it)$ and $D_n=\sum_{k=1}^n\frac{1}{k}$, is $w^*$-convergent to $\mathcal{N}_{0,\sigma_f^2}$. 

The endomorphism $\theta$ satisfies the tracial version. Regarding $S^1$ as $\partial_e(T(A))$, $\mu_0$ is the unique representing measure of a trace $\tau_0=\tau_{\mu_0}\in T(A)$, namely, $\tau_0(a)=\int_{S^1}\hat{a}\,d\mu_0$ for $a\in A_{sa}$. If $a$ is tracially Lipschitz, with $\tau_0(a)=0$ and $\sigma_{\hat{a}}^2>0$ on $\partial_e(T(A))$, then for almost every $\tau\in\partial_e(T(A))$, the sequence of weighted averages (\ref{eqn:weighted}) of the point masses $\left\{\delta_{\frac{1}{\sqrt{k}}\tau(a+\theta(a)+\dots+\theta^{k-1}(a))}\right\}_{k=1}^n$ is $w^*$-convergent to $\mathcal{N}_{0,\sigma_{\hat{a}}^2}$.
\end{proof}


\begin{thebibliography}{10}

\bibitem{Alfsen:1971hl}
E.~M. Alfsen.
\newblock {\em Compact convex sets and boundary integrals}.
\newblock Springer-Verlag, New York, 1971.
\newblock Ergebnisse der Mathematik und ihrer Grenzgebiete, Band 57.

\bibitem{Ambrosio:2003vk}
L.~Ambrosio and A.~Pratelli.
\newblock Existence and stability results in the {$L^1$} theory of optimal
  transportation.
\newblock In {\em Optimal transportation and applications ({M}artina {F}ranca,
  2001)}, volume 1813 of {\em Lecture Notes in Math.}, pages 123--160.
  Springer, Berlin, 2003.

\bibitem{Arvanitoyeorgos:2003vq}
A.~Arvanitoyeorgos.
\newblock {\em An introduction to {L}ie groups and the geometry of homogeneous
  spaces}, volume~22 of {\em Student Mathematical Library}.
\newblock American Mathematical Society, Providence, RI, 2003.
\newblock Translated from the 1999 Greek original and revised by the author.

\bibitem{Bosa:aa}
J.~Bosa, N.~P. Brown, Y.~Sato, A.~Tikuisis, S.~White, and W.~Winter.
\newblock Covering dimension of {$\cs$}-algebras and 2-coloured
  classification.
\newblock {\em Mem. Amer. Math. Soc.}, 257(1233):vii+97, 2019.

\bibitem{Carrion:wz}
J.~R. Carri\'{o}n, J.~Gabe, C.~Schafhauser, A.~Tikuisis, and S.~White.
\newblock Classification of $^*$-homomorphisms {\rm I}: {S}imple nuclear {$\cs$}-algebras.
\newblock {In preparation.}

\bibitem{Castillejos:2021wm}
J.~Castillejos, S.~Evington, A.~Tikuisis, S.~White, and W.~Winter.
\newblock Nuclear dimension of simple {$\cs$}-algebras.
\newblock {\em Invent. Math.}, 224(1):245--290, 2021.

\bibitem{Chazottes:2015ti}
J.-R. Chazottes.
\newblock Fluctuations of observables in dynamical systems: from limit theorems
  to concentration inequalities.
\newblock In {\em Nonlinear dynamics new directions}, volume~11 of {\em
  Nonlinear Syst. Complex.}, pages 47--85. Springer, Cham, 2015.
  
\bibitem{Chazottes:2007wg}
J.-R. Chazottes and S.~Gou\"{e}zel.
\newblock On almost-sure versions of classical limit theorems for dynamical
  systems.
\newblock {\em Probab. Theory Related Fields}, 138(1-2):195--234, 2007.

\bibitem{Cheong:2015aa}
C.~W. Cheong.
\newblock Unitary orbits in 1-dimensional {NCCW} complexes.
\newblock {\em J. Funct. Anal.}, 269(9):2977--2994, 2015.

\bibitem{Carmo:1976um}
M.~P. do~Carmo.
\newblock {\em Differential geometry of curves and surfaces}.
\newblock Prentice-Hall, Inc., Englewood Cliffs, N.J., 1976.
\newblock Translated from the Portuguese.

\bibitem{Eilers:1998yu}
S.~Eilers, T.~A. Loring, and G.~K. Pedersen.
\newblock Stability of anticommutation relations: an application of
  noncommutative {CW} complexes.
\newblock {\em J. Reine Angew. Math.}, 499:101--143, 1998.

\bibitem{Elliott:2016ab}
G.~A. Elliott, G.~Gong, H.~Lin, and Z.~Niu.
\newblock On the classification of simple amenable {$\cs$}-algebras with finite
  decomposition rank, {\rm{II}}.
\newblock arXiv:1507.03437 [math.OA], 2016.

\bibitem{Gibbs:2002aa}
A.~L. Gibbs and F.~E. Su.
\newblock On choosing and bounding probability metrics.
\newblock {\em Int. Stat. Rev.}, 70(3):419--435, 2002.

\bibitem{Givens:1984to}
C.~R. Givens and R.~M. Shortt.
\newblock A class of {W}asserstein metrics for probability distributions.
\newblock {\em Michigan Math. J.}, 31(2):231--240, 1984.

\bibitem{Gong:2020uf}
G.~Gong, H.~Lin, and Z.~Niu.
\newblock A classification of finite simple amenable {$\mathcal Z$}-stable
  {$\cs$}-algebras, {II}: {$\cs$}-algebras with rational
  generalized tracial rank one.
\newblock {\em C. R. Math. Acad. Sci. Soc. R. Can.}, 42(4):451--539, 2020.

\bibitem{Hiai:1989aa}
F.~Hiai and Y.~Nakamura.
\newblock Distance between unitary orbits in von {N}eumann algebras.
\newblock {\em Pacific J. Math.}, 138(2):259--294, 1989.

\bibitem{Hodgkin:1967kq}
L.~Hodgkin.
\newblock On the {$K$}-theory of {L}ie groups.
\newblock {\em Topology}, 6:1--36, 1967.

\bibitem{Hu:2015aa}
S.~Hu and H.~Lin.
\newblock Distance between unitary orbits of normal elements in simple
  {$\cs$}-algebras of real rank zero.
\newblock {\em J. Funct. Anal.}, 269(2):355--437, 2015.

\bibitem{Jacelon:2010fj}
B.~Jacelon.
\newblock A simple, monotracial, stably projectionless {$\cs$}-algebra.
\newblock {\em J. Lond. Math. Soc. (2)}, 87(2):365--383, 2013.

\bibitem{Jacelon:2022wr}
B.~Jacelon.
\newblock Chaotic tracial dynamics.
\newblock arXiv:2205.03951 [math.OA], 2022.

\bibitem{Jacelon:2014aa}
B.~Jacelon, K.~R. Strung, and A.~S. Toms.
\newblock Unitary orbits of self-adjoint operators in simple
  {$\js$}-stable {$\cs$}-algebras.
\newblock {\em J. Funct. Anal.}, 269(10):3304--3315, 2015.

\bibitem{Jacelon:2021wa}
B.~Jacelon, K.~R. Strung, and A.~Vignati.
\newblock Optimal transport and unitary orbits in {$\cs$}-algebras.
\newblock {\em J. Funct. Anal.}, 281(5):109068, 2021.

\bibitem{Jacelon:2021uc}
B.~Jacelon and A.~Vignati.
\newblock Stably projectionless {F}ra\"{i}ss\'{e} limits.
\newblock {\em Studia Math.}, 267(2):161--199, 2022.

\bibitem{Jiang:1999hb}
X.~Jiang and H.~Su.
\newblock On a simple unital projectionless {$\cs$}-algebra.
\newblock {\em Amer. J. Math.}, 121(2):359--413, 1999.

\bibitem{Kantorovic:1957we}
L.~V. Kantorovi\v{c} and G.~Rubin\v{s}te\u{\i}n.
\newblock On a functional space and certain extremum problems.
\newblock {\em Dokl. Akad. Nauk SSSR (N.S.)}, 115:1058--1061, 1957.

\bibitem{Kishimoto:1996yu}
A.~Kishimoto and A.~Kumjian.
\newblock Simple stably projectionless {$\cs$}-algebras arising as crossed
  products.
\newblock {\em Canad. J. Math.}, 48(5):980--996, 1996.

\bibitem{Li:1999aa}
L.~Li.
\newblock Simple inductive limit {$\cs$}-algebras: spectra and approximations
  by interval algebras.
\newblock {\em J. Reine Angew. Math.}, 507:57--79, 1999.

\bibitem{Lin:2011ub}
H.~Lin and Z.~Niu.
\newblock The range of a class of classifiable separable simple amenable
  {$\cs$}-algebras.
\newblock {\em J. Funct. Anal.}, 260(1):1--29, 2011.

\bibitem{Lin:2014aa}
H.~Lin and Z.~Niu.
\newblock Homomorphisms into simple {$\mathcal{Z}$}-stable {$\cs$}-algebras.
\newblock {\em J. Operator Theory}, 71(2):517--569, 2014.

\bibitem{Loring:1993kq}
T.~A. Loring.
\newblock {$\cs$}-algebras generated by stable relations.
\newblock {\em J. Funct. Anal.}, 112(1):159--203, 1993.

\bibitem{Masumoto:2017wx}
S.~Masumoto.
\newblock The {J}iang-{S}u algebra as a {F}ra\"{\i}ss\'{e} limit.
\newblock {\em J. Symb. Log.}, 82(4):1541--1559, 2017.

\bibitem{mathoverflow/steve:tm}
MathOverflow User https://mathoverflow.net/users/106046/steve.
\newblock The infinity {W}asserstein distance {$W_\infty$} and the weak topology.
\newblock MathOverflow.
\newblock URL:https://mathoverflow.net/q/404689 (version: 2021-09-24).

\bibitem{Matui:2011uq}
H.~Matui.
\newblock Classification of homomorphisms into simple {$\js$}-stable
  {$\cs$}-algebras.
\newblock {\em J. Funct. Anal.}, 260(3):797--831, 2011.

\bibitem{Minami:1975wx}
H.~Minami.
\newblock {$K$}-groups of symmetric spaces. {I}.
\newblock {\em Osaka Math. J.}, 12(3):623--634, 1975.

\bibitem{Oxtoby:1941aa}
J.~C. Oxtoby and S.~M. Ulam.
\newblock Measure-preserving homeomorphisms and metrical transitivity.
\newblock {\em Ann. of Math. (2)}, 42:874--920, 1941.

\bibitem{Prolla:1994td}
J.~B. Prolla.
\newblock On the {W}eierstrass-{S}tone theorem.
\newblock {\em J. Approx. Theory}, 78(3): 299--313, 1994.

\bibitem{Razak:2002kq}
S.~Razak.
\newblock On the classification of simple stably projectionless
  {$\cs$}-algebras.
\newblock {\em Canad. J. Math.}, 54(1):138--224, 2002.

\bibitem{Rieffel:1999aa}
M.~A. Rieffel.
\newblock Metrics on state spaces.
\newblock {\em Doc. Math.}, 4:559--600, 1999.

\bibitem{Robert:2010qy}
L.~Robert.
\newblock Classification of inductive limits of 1-dimensional {NCCW} complexes.
\newblock {\em Adv. Math.}, 231(5):2802--2836, 2012.

\bibitem{Robert:2010rz}
L.~Robert and L.~Santiago.
\newblock Classification of {$\cs$}-homomorphisms from {$C_0(0,1]$} to a
  {$\cs$}-algebra.
\newblock {\em J. Funct. Anal.}, 258(3):869--892, 2010.

\bibitem{Rordam:2009qy}
M.~R{\o}rdam and W.~Winter.
\newblock The {J}iang--{S}u algebra revisited.
\newblock {\em J. Reine Angew. Math.}, 642:129--155, 2010.

\bibitem{Thomsen:1994qy}
K.~Thomsen.
\newblock Inductive limits of interval algebras: the tracial state space.
\newblock {\em Amer. J. Math.}, 116(3):605--620, 1994.

\bibitem{Toms:2009uq}
A.~S. Toms.
\newblock Comparison theory and smooth minimal {$\cs$}-dynamics.
\newblock {\em Comm. Math. Phys.}, 289(2):401--433, 2009.

\bibitem{Tsang:2003qy}
K.-W. Tsang.
\newblock {\em A classification of certain simple stably projectionless
  {$\cs$}-algebras}.
\newblock ProQuest LLC, Ann Arbor, MI, 2003.
\newblock Thesis (Ph.D.)--University of Toronto (Canada).

\bibitem{Tsang:2005fj}
K.-W. Tsang.
\newblock On the positive tracial cones of simple stably projectionless {$\cs$}-algebras.
\newblock {\em J. Funct. Anal.}, 227(1):188--199, 2005.

\bibitem{Villani:2009aa}
C.~Villani.
\newblock {\em Optimal transport}, volume 338 of {\em Grundlehren der
  Mathematischen Wissenschaften [Fundamental Principles of Mathematical
  Sciences]}.
\newblock Springer-Verlag, Berlin, 2009.
\newblock Old and new.

\bibitem{Winter:2012pi}
W.~Winter.
\newblock Nuclear dimension and {$\js$}-stability of pure
  {$\cs$}-algebras.
\newblock {\em Invent. Math.}, 187(2):259--342, 2012.

\bibitem{Young:1999vu}
L.-S. Young.
\newblock Recurrence times and rates of mixing.
\newblock {\em Israel J. Math.}, 110:153--188, 1999.

\end{thebibliography}
\end{document}